%% file: simplices.tex
\documentclass[12pt]{article}

\usepackage[utf8]{inputenc}
\usepackage[T1]{fontenc}

\usepackage[a4paper, vmargin=2cm, hmargin=2cm]{geometry}
\usepackage{ragged2e}
\usepackage{sectsty}
\sectionfont{\centering}

\usepackage[hidelinks]{hyperref}
\hypersetup{
	colorlinks=true, 
	citecolor = blue, 
	linkcolor = blue, 
	urlcolor = blue, 
	anchorcolor = blue
}

\usepackage{graphicx}
\usepackage{standalone}
\usepackage{pgf,tikz,pgfplots}
\usetikzlibrary{arrows,patterns}
\pgfplotsset{compat=1.18}
\tikzset{every picture/.style={font=\normalsize}}

\usepackage{amsmath, amsthm, amsfonts}
\usepackage{amscd, amssymb, latexsym}
\usepackage{cleveref}

\usepackage{comment}
\usepackage{xcolor}
\usepackage{enumitem}

\theoremstyle{plain}
\newtheorem{theorem}{Theorem}[section]
\newtheorem{corollary}[theorem]{Corollary}
\newtheorem{lemma}[theorem]{Lemma}
\newtheorem{proposition}[theorem]{Proposition}

\theoremstyle{definition}

\newtheorem{example}[theorem]{Example}
\newtheorem{problem}[theorem]{Problem}
\theoremstyle{remark}
\newtheorem{remark}[theorem]{Remark}

\newcommand{\N}{\mathbb{N}}

\newcommand{\R}{\mathbb{R}}
\newcommand{\C}{\mathbb{C}}

\newcommand{\Hyp}{\mathbf{H}}
\newcommand{\Sph}{\mathbf{S}}
\newcommand{\Ball}{\mathbf{B}}
\newcommand{\Proj}{\mathbf{P}}

\DeclareMathOperator{\Isom}{Isom}

\DeclareMathOperator{\GO}{O}

\DeclareMathOperator{\Span}{Span}
\DeclareMathOperator{\Conv}{Conv}
\DeclareMathOperator{\sign}{sign}

\title{Hyperbolic Simplices of Maximal Inradius}

\author{Bruno DUCHESNE, Christopher-Lloyd SIMON}
\date{\today}

\begin{document}
	
	\maketitle
	
	\begin{abstract}
		For $n\in \mathbb{N}$, consider a hyperbolic $n$-dimensional simplex $\Delta$, defined by $1+n$ points in the compactified hyperbolic space $\mathbf{H}^n \sqcup \partial \mathbf{H}^n$. For each integer $m\le n$, denote $\delta^n_m(\Delta)\in [0,+\infty]$ the Hausdorff distance between its skeleta of dimensions $n$ and $m$. In particular, $\delta^n_{n-1}(\Delta)$ is its inradius. The maximum of $\delta^n_m(\Delta)$ over $\Delta\in (\mathbf{H}^n \sqcup \partial \mathbf{H}^n)^{1+n}$ is denoted $\mu^n_m\in [0,+\infty]$.
		
		We first show that $\Delta$ has maximal inradius $\delta^n_{n-1}(\Delta)=\mu^n_m$ if and only if its is (total) ideal and regular; for which the inradius is given by $\tanh \mu^n_{n-1} = 1/n$. We deduce that $\Delta$ has maximal $\delta^n_{n-1}(\Delta)=\mu^n_m$ if and only if it is (total) ideal and regular. We compute that the maximal distance to the $1$-skeleton $\mu^n_1$ is given by $\left(\tanh \mu^n_1\right)^2 = (n-1)/(2n)$ and deduce that those are uniformly bounded by $\lim_{n} \mu^n_1 = \log(1+\sqrt{2})$.
	\end{abstract}
	
	\section{Introduction}
	
	For $n\in \N_{\ge 2}$, let $\Hyp^n \sqcup \partial \Hyp^n$ be the compactified real algebraic hyperbolic space of dimension $n$.
	We will recall its definition and main properties in Section \ref{sec:geometry-Hn}.
	
	Here, a \emph{simplex} $\Delta$ is defined by a sequence of $1+n$ points in $\Hyp^n \sqcup \partial \Hyp^n$ called its \emph{vertices}.
	It is \emph{total} when its vertices are independent.
	It is \emph{ideal} when all its vertices belong to $\partial \Hyp^n$.
	
	For $m\in \{0,\dots, n\}$, a $m$-face of $\Delta$ is the convex hull of $1+m$ vertices. 
	Its $m$-skeleton $\Delta^{(m)}$ is the union of its $m$-faces.
	For $n>m>0$, we study the Hausdorff distance from $\Delta^{(n)}$ to $\Delta^{(m)}$:
	\(\delta^n_m(\Delta) = \max\{d(p, \Delta^{(m)}) \colon p \in \Delta^{(n)} \}\).
	In particular $\delta^n_{n-1}(\Delta)$ is the \emph{inradius} of $\Delta$ (\Cref{subsec:incenter}).
	
	\begin{figure}[h]
		\centering
		\includegraphics[width=0.4231\linewidth]{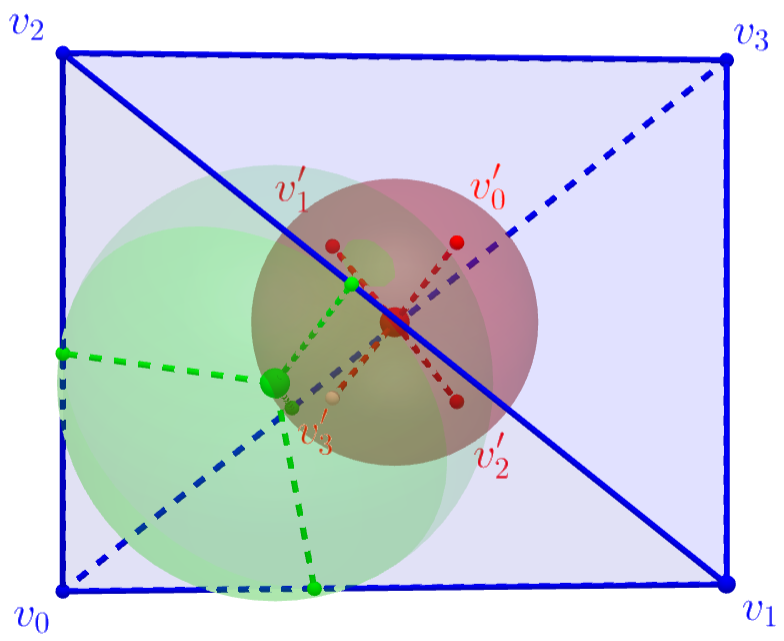}
		\includegraphics[width=0.4231\linewidth]{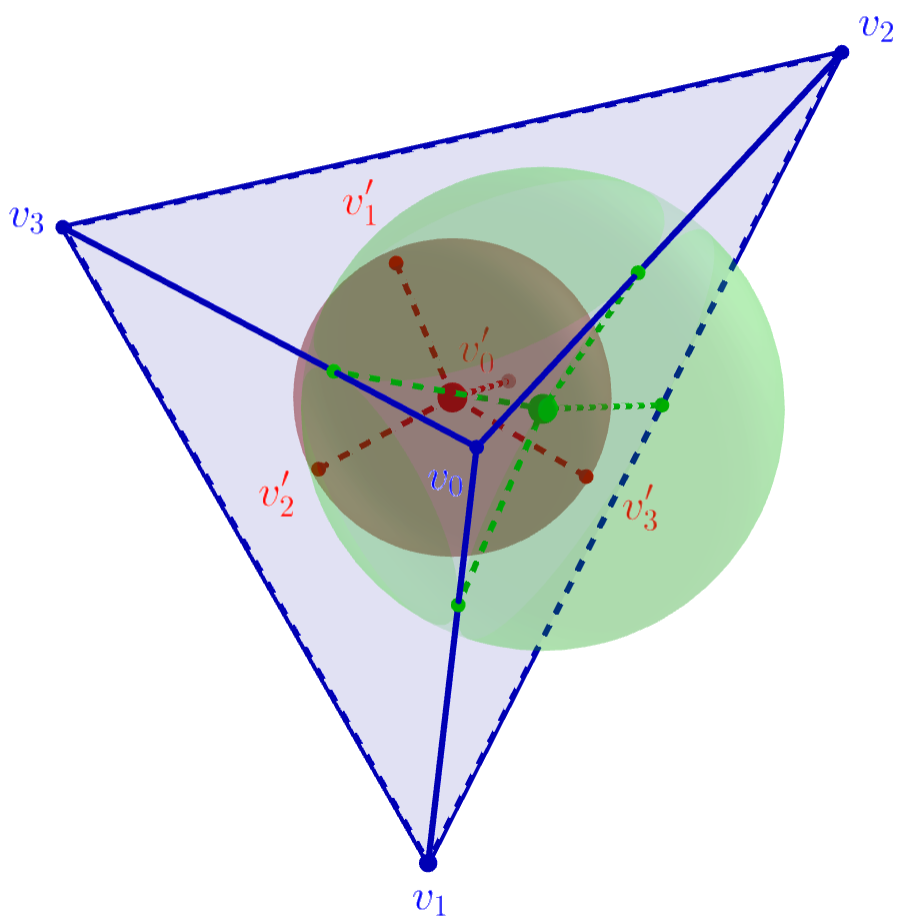} 
		\caption{Visualizing \textcolor{green!50!black}{$\delta^3_2$} and \textcolor{red}{$\delta^3_1$} in the incentred Euclidean model for an ideal simplex of $\Hyp^3$.}
		\label{fig:H3-ideal-simplex}
	\end{figure}
	
	\Cref{subsec:incenter} introduces the incentred Euclidean model of a hyperbolic simplex, establishing a correspondence between ideal simplices and Euclidean simplices inscribed in the unit sphere whose incenter lies at the origin, which is equivariant under $\Isom(\Hyp^n)$ and $ \Isom(\Sph^{n-1})$.
	
	A total simplex of $\Hyp^n\sqcup \partial \Hyp^n$ is \emph{regular} when its stabilizer in $\Isom(\Hyp^n)$ is isomorphic to the symmetric group $\mathfrak{S}_{n+1}$ of its vertices.
	There exists a total ideal simplex of $\Hyp^n\sqcup \partial \Hyp^n$ which is regular, it is unique up to the action of $\Isom(\Hyp^n)$.
	
	\begin{theorem}[maximal inradius]
		\label{thm_intro:ideal-simplices-max-inradius}
		A total simplex $\Delta$ of $\Hyp^n\sqcup \partial \Hyp^n$ has inradius $\delta^n_{n-1}(\Delta)\le \tanh^{-1}(1/n)$, with equality if and only if it is ideal and regular.
	\end{theorem}
	
	\begin{proof}[Proof outline]
		We work in the incentred Euclidean model, and use barycentric coordinates to show that Euclidean simplices with maximal inradius admit an orthocenter that coincides with the incenter.
		The regularity then follows from \cite[Theorem 4.3]{Edmonds-Allan-Martini_Orthocentric-simplices-centers_2005}.
	\end{proof}
	
	\begin{remark}[maximal inradius]
		We have no reference questioning or probing the characterization of hyperbolic simplices with maximal inradius, even though the inradius is closely studied or extensively used in several works \cite{Jacquemet_inradius-hyperbolic-simplex_2014, Peyrimohoff_simplices-minimal-edge-length-hyperbolic_2002}.
	\end{remark}
	
	\begin{remark}[maximal volume]
		The simplices of $\Hyp^n\sqcup \partial \Hyp^n$ with maximal volume are precisely those which are total ideal and regular.
		This was conjectured by Thurston and a complete proof was first given by \cite{Haagerup-Munkholm_hyperbolic-simplices-max-vol_1981} using analytic methods, then by \cite{Peyrimohoff_simplices-minimal-edge-length-hyperbolic_2002} using Steiner symmetrization.
		We wonder if our \Cref{cor:incentred-Euclidean-model} and the method in Theorem \ref{thm:ideal-simplices-max-inradius} could yield a simpler proof.
	\end{remark}
	
	For $n,m\in \N$, on the compact space $(\Hyp^n\sqcup\partial\Hyp^n)^{n+1}$ of hyperbolic simplices in dimension $n$, the function $\delta^n_m$ is continuous so it admits a maximum $\mu^n_m =\max_\Delta \delta^n_m(\Delta)\in [0,\infty]$, which must be achieved on a closed subset of total simplices which are ideal.
	%
	
	\begin{theorem}[maximal Hausdorff distance to $m$-skeleton]
		\label{thmIntro:Hausdist-simplex-skeleta}
		For integers $n>m > 0$, a simplex $\Delta$ in $\Hyp^n\sqcup \partial \Hyp^n$ satisfies $\delta^n_m(\Delta)= \mu^n_m$ if and only if $\Delta$ is total, ideal and regular.
		Moreover, $\mu^{n}_{1}$ is given 
		by \(\left(\tanh \mu^n_1\right)^2 = \tfrac{n-1}{2n}\) so $\mu^n_1$ grows as $n\to \infty$ to $\mu^\infty_1 = \tanh^{-1}(1/\sqrt{2}) = \log(1+\sqrt{2})$.
	\end{theorem}
	
	\begin{proof}[Proof outline]
		For a total ideal simplex $\Delta$, we consider the successive projections of $p\in \Delta^{(n-c)}$ on the closest faces of increasing codimension $c$: applying the hyperbolic Pythagorean theorem and the inequality in \Cref{thm_intro:ideal-simplices-max-inradius} shows that equality is achieved only for regular $\Delta$.
	\end{proof}
	
	\begin{remark}[bounding $\mu^n_1$]
		The finiteness of $\mu^n_1$ is mentioned in the proof of \cite[Theorem 3]{Bestvina_degenerations-hyperbolic-space_1988} with a reference to \cite{Bonahon-bouts-des-varietes-hyperboliques-de-dimension-3}, but our proof is complete and simple.
		
		In \cite[Theorem 3]{Bestvina_degenerations-hyperbolic-space_1988}, the finiteness of $\mu^n_1$ is used to show certain results about the convergence of group actions on $\Hyp^n$ for the equivariant Gromov-Hausdorff topology.
		We were interested (and surprised) by fact that $\mu^n_1$ is bounded by $\mu^\infty_1$, as this enabled us to extend these results in \cite{BD-CLS_actions-Hinfini_2026} and show Gromov-Hausdorff convergence for sequences of group actions on $\Hyp^{n_i}$ in a setting where $n_i\in \N\cup\{\infty\}$ is any sequence of countable cardinals.
	\end{remark}
	
	
		
		
		We finish with the following \Cref{prob:maximizers_delta-n-m} concerning the enumeration and localization of maximizers of $\delta^n_m(\Delta)$, and the \Cref{eg:disphenoids} of disphenoids.
		
		\begin{problem}[maximizers of $\delta^n_m$]
			Each $(n+1)$-set of $m$-faces $\{F_{0}^{(m)}, \dots, F_n^{(m)}\} \in \left(\genfrac{}{}{0pt}{}{\binom{\Delta}{1+m}}{1+n}\right)$
			defines a decreasing intersection of convex sets in $\Delta^{(n)}$ by $C_t(F)=\{p\in \Delta^{(n)} \colon d(p,F_i^{(m)}\ge t\}$, hence a point maximizing the distance to these faces, that is a local maximizer of the Hausdorff distance from $\Delta^{(n)}$ to $\Delta^{(m)}$, which for some $F$ will be a global maximer of $\delta^n_m$. 
			
			\emph{What can be the number of distinct local global minimizers and where can they be located?}
			
			Note that the isometry group of the simplex $\operatorname{Stab}(\Delta^{(n)})\subset \Isom(\Hyp^n)$ acts on that set of $(1+n)$-sets of $m$-faces.
			However the description and enumeration of distinct (local or global) maximizers is more subtle than the question of classifying (all or certain) of the orbits under this action, since it may happen that two $(1+n)$-sets of $m$-faces yield the same (local or global) maximizer without belonging to the same orbit.
		\end{problem}
		
		\section{Geometry of the algebraic hyperbolic space}
		\label{sec:geometry-Hn}
		
		In this section, after recalling some geometry of the Lorentz-Minkowski space time, we define the compactified real algebraic hyperbolic space of dimension $n\in \N_{>0}$ denoted $\Hyp^n\sqcup \partial \Hyp^n$ by constructing three models, and use each of them to discuss some of its properties.
		
		
		The Minkowski hyperboloid model is suited to compute metric quantities (angles, distances and volumes) in $\Hyp^n$ using the algebra of the scalar product.
		The Cayley-Klein projective model is adapted for drawing and reasoning with incidence and orthogonality relations between points, lines and quadrics.
		We will use the Euclidean ball model to work around a distinguished point (the incenter of a simplex) as it provides a conformal Euclidean model of its tangent space.
		
		\subsection{Geometry of the Lorentz-Minkowski space-time}
		\label{subsec:Lorentz-Minkowski}
		Fix $n\in \N$. The Minkowski space of dimension $1+n$ is the oriented $\R$-vector space $\R^{1+n}$ with the symmetric bilinear form of signature $(1,n)$ defined for $x,y\in\R^{1+n}$ by \(\langle x,y\rangle =x_0v_0-\sum_{i>0}x_iv_i\).
		The \emph{Gram matrix} of a sequence $(w_0,\dots,w_l)\in (\R^{1+n})^l$ is defined by $G(u)=(\langle w_i \mid w_j\rangle)_{i,j}$. This is a real symmetric matrix of signature $(s_+,s_0,s_-)$ where $s_+\le 1$ and $s_-\le n$.
		An \emph{orthonormal frame} is a basis of $\R^{1+n}$ whose Gram matrix is $\operatorname{diag}(1,-1,\dots,-1)$.
		
		For $r\in \R$, let $\mathbf{X}_{r} =\{x\in\R^{1+n} \colon \langle x,x\rangle=r\}$ be the pseudo-sphere of radius $r$ and denote by $\mathbf{Y}_r=\{x\in \mathbf{X}_r \colon x_0\ge 0\}$ its intersection with the upper-half space.
		The isotropic cone $\mathbf{X}_{0}$ is asymptotic to the $2$-sheeted hyperboloid $\mathbf{X}_{+1}$ and to the $1$-sheeted hyperboloid $\mathbf{X}_{-1}$.
		In the projective space $\Proj{\R^{1+n}}$, the non-degenerate quadric $\Proj{\mathbf{X}_0}$ homeomorphic to the sphere $\Sph^{n-1}$ has inside $\Proj{\mathbf{X}_{+1}}$ homeomorphic to the open ball $\Ball^n$ and outside $\Proj{\mathbf{X}_{-1}}$ homeomorphic to $\Sph^{n-1}\times \R$.
		
		\begin{figure}[h]
			\centering
			\includegraphics[width=0.48\linewidth]{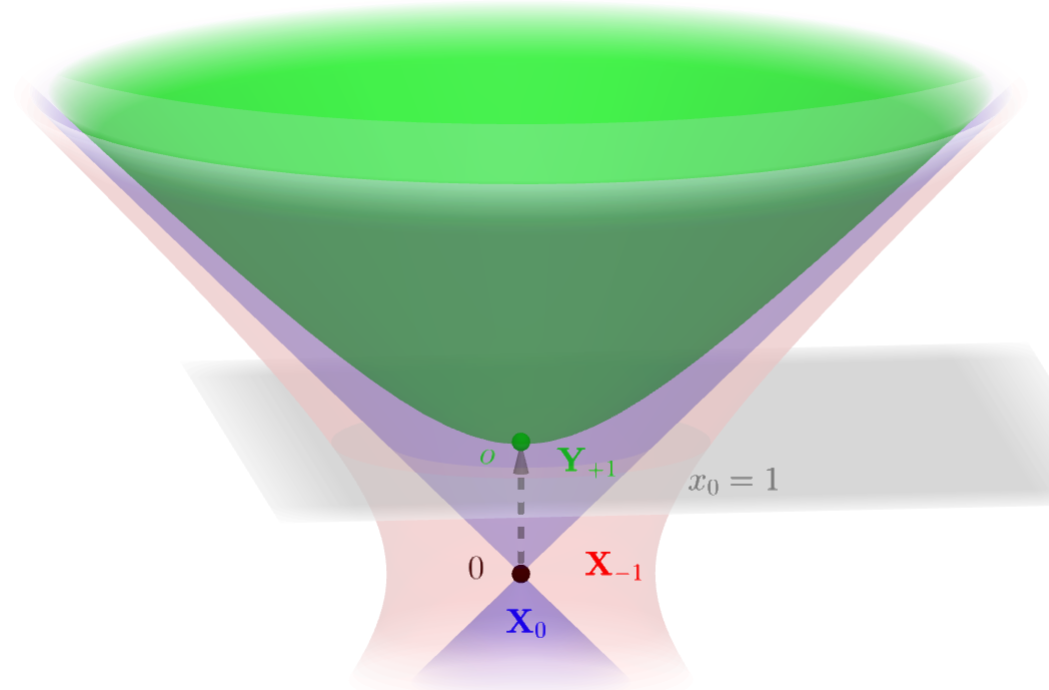}
			\hfill
			\includegraphics[width=0.38\linewidth]{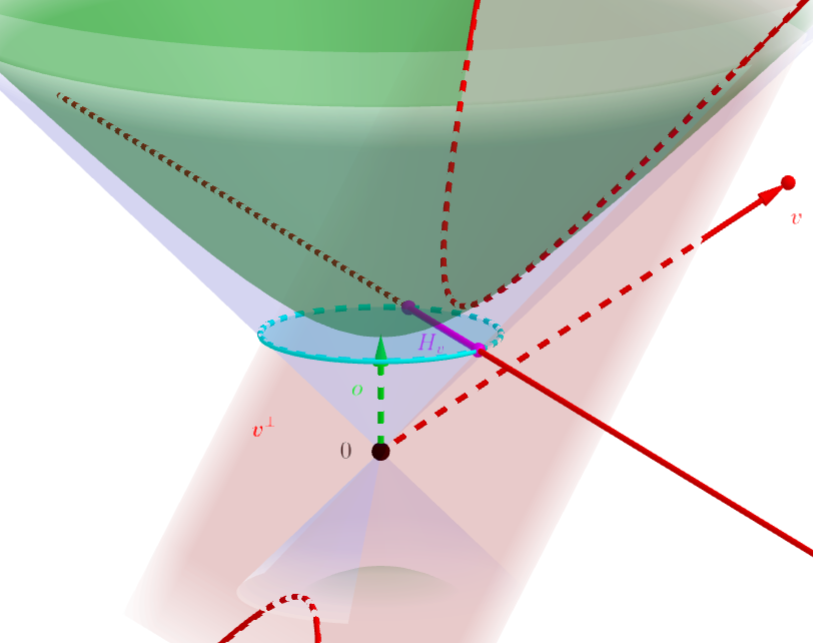}
			\caption{In Minkowski space $\R^{1+2}$: (upper-half) quadric surfaces $\mathbf{Y}_r \subset \mathbf{X}_r$ for $r=\{-1,0,1\}$.\\
				The orthogonality relation projectivizes to polarity with respect to the quadric $\Proj{\mathbf{X}_0}$.}
			\label{fig:quadric-surfaces-polarity}
		\end{figure}
		
		The group $\GO(1,n)$ acts transitively on $\mathbf{X}_{+1}$, and the connected component of the identity $\GO^+(1,n)$ acts transitively on $\mathbf{Y}_{+1}$.
		The stabiliser of the point $o=(+1,0,\dots,0)$ is the maximal compact connected subgroup $\GO(n)$.
		More generally, $\GO(1,n)$ acts transitively on bases of $\R^{1+n}$ with a given Gram matrix: the unique linear matrix sending one basis to the other must preserve the quadratic form. 
		It follows that $\GO(1,n)$ acts transitively on sequences of $l\in \N$ vectors in $\R^{1+n}$ with a given Gram matrix.
		
		For $u\in \R^{1+n}\setminus\{0\}$, its orthogonal hyperplane denoted $u^\perp = \{x\in \R^{1+n} \colon \langle u\mid x\rangle =0\}$ is naturally co-oriented as the boundary of the closed half-space $u^{+} = \{x\in \R^{n+1} \colon \langle u\mid x\rangle \ge 0\}$.
		Since $\langle\mid \rangle$ is non-degenerate, the orthogonality $u \mapsto u^\perp$ projectifies to an isomomorphism between the projective space $\Proj{(\R^{1+n})}$ and its dual $\Proj{(\R^{1+n})}^*$, which is double covered by the isomorphism induced by $u\mapsto u^+$ between the Grassmanians of oriented lines $\Proj^+{(\R^{1+n})}$ and of co-oriented hyperplanes $\Proj^+{(\R^{1+n})}^*$.
		This isomorphism called \emph{polarity} with respect to $\Proj{\mathbf{X}_0}$ or $\Proj^+{\mathbf{X}_0}$ can be constructed by straight lines and incidence relations between point, lines and the quadrics.
		
		In $(\R^{1+n},\langle \mid \rangle)$, a basis $(v_i)$ has a unique dual basis $(v_j^\star)$ defined by $(\langle v_i^\star \mid v_j\rangle)_{ij} = \operatorname{Id}_{1+n}$.
		The decomposition of $p\in \R^{1+n}$ as $p= \sum_0^n \lambda_k v_k$ has coordinates $\lambda_k = \langle p \mid v_k^\star \rangle$ also given by:
		\begin{equation*}
			\lambda_k = \frac{\det G(v_0,\dots, v_{k-1}, p, v_{k+1}, \dots, v_n)}{\det G(v_0,\dots, v_{k-1}, v_k, v_{k+1}, \dots, v_n)} 
		\end{equation*}
		Indeed, $\lambda_k$ is the unique linear form on $\R^{1+n}$ which vanishes on $\{v_0,\dots,v_n\}\setminus\{v_k\}$ and takes the value $1$ at $v_k$, and the same goes for the scalar product with $v_k^\star$ and that ratio of determinants.
		
		\subsection{The Minkowski hyperboloid model}
		
		The \emph{Minkowski hyperboloid model}
		of $\Hyp^n$ is given by endowing $\mathbf{Y}_{+1}$ with the Riemannian metric obtained by restricting the form $-\langle \mid \rangle$ to its tangent spaces.
		For example, the tangent space at $o=\mathbf{Y}_{+1}\cap\{x\in \R^{1+n} \colon x_0=1\}$ is the horizontal hyperplane $\R^n$ with Euclidean metric $-\langle \mid \rangle$.
		
		For $k\in \{1,\dots, n\}$, the geodesic subspaces of dimension $k$ in $\Hyp^n=\mathbf{Y}_{+1}$ are its (non-empty) intersections with the linear spaces of dimension $1+k$ in $\R^{1+n}$ (to which the restriction of the form $\langle \mid \rangle$ has signature $(k,1)$).
		Hence the notions of linear rank, span and independence in $\R^{1+n}$ restrict to the corresponding notions of geodesic rank, span and independence in $\Hyp^n$.
		
		For $v_0,\dots,v_m\in \mathbf{Y}_{+1}$, their geodesic convex hull is obtained by intersecting $\mathbf{Y}_{+1}$ with the positive cone over their linear convex hull $\R_+\cdot \Conv(v_0,\dots,v_m)$, or equivalently (by convexity of $\mathbf{Y}_{+1}$) with their convex hull with the origin $\Conv(0,v_0,\dots,v_m)$.
		This is a hyperbolic polytope whose geodesic boundary consists of portions of the hyperboloid.
		
		The \emph{Lorentzian sphere model} for $\partial \Hyp^n$ namely the sphere of asymptotic directions of the hyperbolic space can be represented by any section of the projectivization map $\mathbf{Y}_0\to \Proj{\mathbf{Y}_0}$, such as $\mathbf{Y}_0\cap \{x\in \R^{1+n} \colon x_0=+1\}$. 
		(The compactification will appear in the next \Cref{subsec:Cayley-Klein}.)
		
		We define $(\Hyp^n)^\star$ as the set of hyperspaces (geodesic subspaces of codimension $1$), and $(\Hyp^n)^+$ as the set of closed half-spaces.
		For $u\in \mathbf{X}_{<0}$ of negative norm, we denote the geodesic hyperspace $H_u=u^{\perp}\cap \mathbf{Y}_{+1}$ and the closed half-space $H_u^+=u^{+}\cap \mathbf{Y}_{+1}$.
		In restriction to $u\in \mathbf{X}_{-1}$, the maps $u\mapsto H_u$ and $u\mapsto H_u^+$ are two-to-one and one-to-one parametrizations of $(\Hyp^n)^\star$ and $(\Hyp^n)^+$.
		
		We now use this Minkowski hyperboloid model to express the distances and angles between points and hyperspaces of $\Hyp^n$ using the algebra of the scalar product.
		
		For $x,y\in \mathbf{Y}_{+1}$, their distance is obtained by integrating the Riemannian metric along the segment of hyperbola $\R_+\cdot\Conv(x,y)\cap \mathbf{Y}_{+1}$ yielding \(\cosh(d(x,y))= \langle x \mid y\rangle\).
		
		For distinct $x,y\in \mathbf{Y}_{+1}$, the difference $u=x-y$ has negative norm and the hyperspace $H_{u}=u^\perp \cap \mathbf{Y}_{+1}$ consists of the set of points equidistant to $x$ and $y$.
		For $x\in \mathbf{Y}_{+1}$ and $u\in \mathbf{X}_{-1}$ the distance from $x$ to $H_u$ is given (in \cite[§3.2]{Ratcliffe_Foundations-Hyperbolic-Manifolds_2019}) by \(\sinh d(x, H_u) = \lvert \langle x, u \rangle\rvert\) and the $\operatorname{sign}(\langle u\mid x \rangle) \in \{-1,0,+1\}$ tells respectively whether $x$ lies in $H_u^+\setminus H_u$ or $H_u$ or $\Hyp^n\setminus H_u^+$.
		For non-proportional $u,v\in \mathbf{X}_{-1}$, the relative positions of the hyperspaces $H_u$ and $H_v$ in $\Hyp^n\sqcup \partial \Hyp^n$ is given (in \cite[§3.2]{Ratcliffe_Foundations-Hyperbolic-Manifolds_2019}) by the $\operatorname{sign}(1-\langle u, v \rangle^2)\in \{+1,0,-1\}$ according to:
		\begin{enumerate}[noitemsep, align=left]
			\item[$\lvert \langle u, v \rangle \rvert < 1$:] $H_u$ and $H_v$ intersect in $\Hyp^n$ at a dihedral angle given by
			\(\cos \angle(u^\perp, v^\perp) = \langle u, v \rangle\);
			
			\item[$\lvert \langle u, v \rangle \rvert = 1$:] $H_u$ and $H_v$ are disjoint but have exactly one common asymptotic direction in $\partial \Hyp^n$; 
			
			\item[$\lvert \langle u, v \rangle \rvert > 1$:] $H_u$ and $H_v$ are separated in $\Hyp^n$ by a positive distance, equal to the length of the unique geodesic arc orthogonal to both $H_u$ and $H_v$, given by \(\cosh d(H_u, H_v) = \lvert \langle u \mid v \rangle\rvert\) . 
		\end{enumerate}
		In each case, the $\operatorname{sign}(\langle u\mid v \rangle) \in \{-1,+1\}$ determines the relative position of the half-spaces $H_u^+$ and $H_u^-$.
		The difference $w=(u-\sign(\langle u\mid v\rangle)v)$ has negative norm and $H_{w}=w^\perp \cap \mathbf{Y}_{+1}$ is a geodesic hyperspace of points equidistant to $H_w$ and $H_v$; more precisely we have respectively:
		\begin{enumerate}[noitemsep, align=left]
			\item[$\lvert \langle u, v \rangle \rvert > 1$:] $H_{u-v}$ and $H_{u+v}$ are the two bisectors of $H_u$ and $H_v$ at their intersection
			\item[$\lvert \langle u, v \rangle \rvert = 1$:] $H_{u-v}$ is equidistant to $H_u$ and $H_v$ (so it shares their common asymptotic direction)
			\item[$\lvert \langle u, v \rangle \rvert < 1$:] $H_{u-v}$ is orthogonal to the orthogeodesic of $H_u$ and $H_v$ at its midpoint
		\end{enumerate}
		
		\begin{figure}[h]
			\centering
			\includestandalone{images/tikz/config-hyperplanes-positive}
			\includestandalone{images/tikz/config-hyperplanes-null}
			\includestandalone{images/tikz/config-hyperplanes-negative}
			\caption{Configurations of hyperplanes in $\Hyp^2$ and their bissectors (see \cref{fig:cross-ratio_polar-incidence_orthoproj-bissect}).}
			\label{fig:config-hyperplanes-Hn}
		\end{figure}
		
		In this model, the group $\Isom(\mathbf{Y}_{+1})$ consists of the two connected components of $\GO(1,n)$ preserving the sign of the first coordinate: it acts transitively on the orthogonal frame bundle of $\mathbf{Y}_{+1}$, hence any metric property about an orthonormal frame at a point holds for all such.
		
		\subsection{The Cayley-Klein projective model}
		\label{subsec:Cayley-Klein}
		The \emph{Cayley-Klein projective model} for $\Hyp^n$ is the subset of lines inside the cone $\Proj\mathbf{X}_{> 0}$.
		The boundary $\partial \Hyp^n$ is the projectivised isotropic cone $\Proj{\mathbf{X}_{0}}$.
		Their union $\Proj{(\mathbf{X}_{\ge 0})}$ defines the topology on the compactified hyperbolic space $\Hyp^n\sqcup \partial \Hyp^n$.
		The space of geodesics hyperspaces $(\partial \Hyp^n)^\star$ corresponds by polarity with respect to the quadric to the points outside of the quadric $\Proj{\mathbf{X}_{-1}}$.
		Hence every point of $\Proj(\R^{1+n})=\Proj{(\mathbf{X}_{+1})} \sqcup \Proj{(\mathbf{X}_{-1})} \sqcup \Proj{(\mathbf{X}_0)}$ has a geometric interpretation and this defines the topology (and geometry) on the extended hyperbolic plane $(\Hyp^n)\sqcup (\partial\Hyp^n) \sqcup (\Hyp^n)^\star$.
		
		The projectivization $\mathbf{X}_{> 0}\to \Proj{\mathbf{X}_{+1}}$ is a homeomorphism in restriction to $\mathbf{Y}_{+1}$, so imposing it to be an isometry leads to defining the distance distance between $[x],[y]\in \Proj{\mathbf{X}_{+1}}$ as the hyperbolic angle:
		\(\cosh d([x],[y])=\langle x,y\rangle/\sqrt{\langle x,x\rangle\langle y,y\rangle}\).
		We may also express this distance in purely projective terms using the cross-ratio of an aligned quadruple of distinct points $(x',x,y,y')$ given in any affine chart containing them by $[x',x,y,y'] = \frac{(x'-y)(x-y')}{(x'-x)(y-y')}$.
		For distinct points $x,y\in \Proj{\mathbf{X}_{+1}}$, the line $(xy)$ intersects the quadric $\Proj\mathbf{X}_0$ in two other points $\{x',y'\}$ such that $(x',x,y,y')$ lie on this order on the line $(xy)$ and we have $\exp 2d(x,y)=\lvert [x',x,y,y'] \rvert$.
		
		In this model, the isometry group $\Isom(\Hyp^n)$ is identified with $\Proj{\GO(1,n)}$.
		
		\begin{figure}[h]
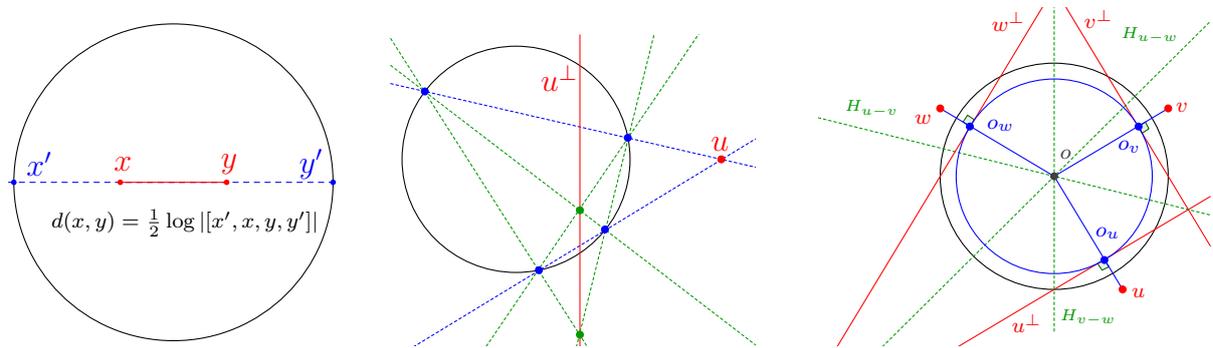

			\centering
			\includestandalone{images/tikz/distance-from-cross-ratio}
			\includestandalone{images/tikz/polar-from-incidence}
			\includestandalone{images/tikz/orthoproj-bissectors}
			\caption{Distance from cross-ratio. Polar from incidence. Orthogonal projections and bissectors.}
			\label{fig:cross-ratio_polar-incidence_orthoproj-bissect}
		\end{figure}
		
		The Cayley-Klein projective model of the extended $\Hyp^n \sqcup \partial \Hyp^n \sqcup (\Hyp^n)^\star$ is especially adapted to draw and reason with incidence and orthogonality relations (between points, lines and the quadric), through straight lines constructions (see Figures \ref{fig:config-hyperplanes-Hn} and \ref{fig:cross-ratio_polar-incidence_orthoproj-bissect}).
		One may also compare distances geometrically (but their analytic and algebraic manipulation are less amenable).
		
		Indeed, for $k\in \{1,\dots, n\}$, the geodesic subspaces of dimension $k$ in $\Hyp^n=\Proj{\mathbf{X}_{+1}}$ are the intersections with the projective linear spaces of dimension $k$ in $\Proj\R^{1+n}$.
		Hence the notions of projective linear rank, span and independence in $\Proj(\R^{1+n})$ restrict to the corresponding notions of geodesic rank, span and independence in $\Hyp^n\sqcup \partial \Hyp^n$.
		
		For $x\in \Proj(\mathbf{X}_{+1})$ and $u\in \Proj(\mathbf{X}_{-1})$, the orthogonal projection of the point $x\in \Hyp^n$ on the hyperspace $H_u\subset \Hyp^n$ is obtained as the intersection of the line $(xu)$ with the hyperspace $u^\perp$.
		For $u\in \Proj^+(\mathbf{X}_{-1})$, one may visualize the boundary $\partial H_u = u^\perp \cap \Proj(\mathbf{X}_0)$ as the set of tangency points between the quadric $\Proj(\mathbf{X}_0)$ and the lines (or hyperplanes) in the pencil based at $u$.
		The half-spaces $H_u^+\subset \Hyp^n$ hyperplanes correspond by duality to the signed-points outside the quadric $u\in \Proj^+(\mathbf{X}_{-1})=\Proj(\mathbf{X}_{-1})\times \{-1,+1\}$ (in bijection with $\mathbf{X}_{-1}$).
		
		\subsection{The Euclidean ball model}
		
		The \emph{Euclidean ball model} is obtained by endowing the affine chart of $\Proj{(\R^{1+n})}$ given by the horizontal affine hyperplane at unit height $o+\R^{n}=\{x\in \R^{1+n} \colon x_0=+1\}$ with the structure of a Euclidean plane with origin $o$ and scalar product $(x,y)\mapsto -\langle x-o \mid y-o \rangle$.
		This yields a homeomorphism between the quadric pair $(\Proj{\mathbf{X}_{0}}, \Proj{\mathbf{X}_{+1}})$ and $(\Sph^{n-1}, \Ball^{n})$, which is equivariant under the actions of $\GO(n)$ viewed as the stabiliser of $o$ in $\GO(n)\subset \Isom(\Hyp^n)$ and as $\Isom(\Ball^n)$.
		
		The points outside the closed ball $u \in (o+\R^n)\setminus (\Ball^{n}\sqcup \Sph^{n-1})$ correspond to the hyperspaces $H_u$ that do not pass through $o$.
		This model inherits all the straight-line constructions from the Cayley-Klein projective model which do not involve hyperspaces through $o$.
		
		For $k\in \{1,\dots, n\}$, the geodesic subspaces of dimension $k$ in $\Hyp^n=\Ball^n$ are the intersections with the linear spaces of dimension $k$ in $o+\R^{n}$.
		Hence the notions of projective linear rank, span and independence in $\R^{n}$ restrict to the corresponding notions of geodesic rank, span and independence in $\Hyp^n\sqcup \partial \Hyp^n$.

		Since $\Isom(\Hyp^n)$ acts freely transitively on the bundle of orthonormal frames to $\Hyp^n$, any metric property about an orthonormal frame holds for all such.
		The Euclidean ball model is well adapted to work in an orthonormal frame at $o$.
		
		In this model, the tangent space of $\Hyp^n=\Ball^n$ at the center $o$ with its induced inner product is isometric to that of $\R^n=\{x_0=0\}$ with its Euclidean inner product $-\langle \mid \rangle$.
		Hence about the center $o$, the hyperbolic angles are conformal to the Euclidean angles, and the hyperbolic sphere $\Sph_o(r)$ of radius $r\in [0,\infty]$ coincides with the Euclidean sphere of radius $\tanh(r)\in [0,1]$.
		
		\begin{figure}[h]
			\centering
			\includegraphics[width=0.28\linewidth]{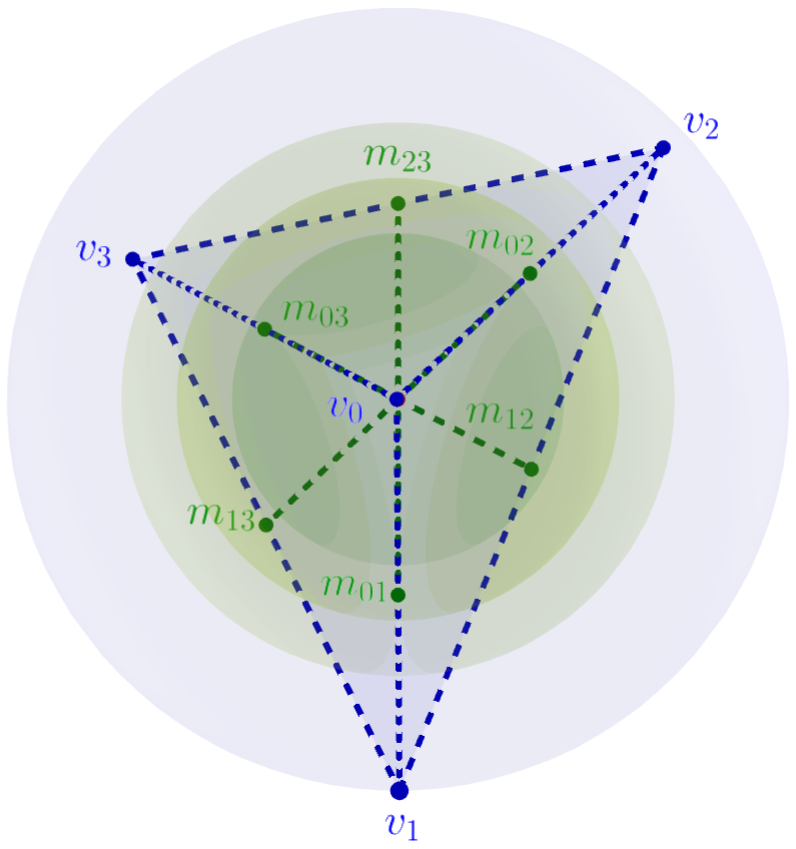}
			\includegraphics[width=0.3\linewidth]{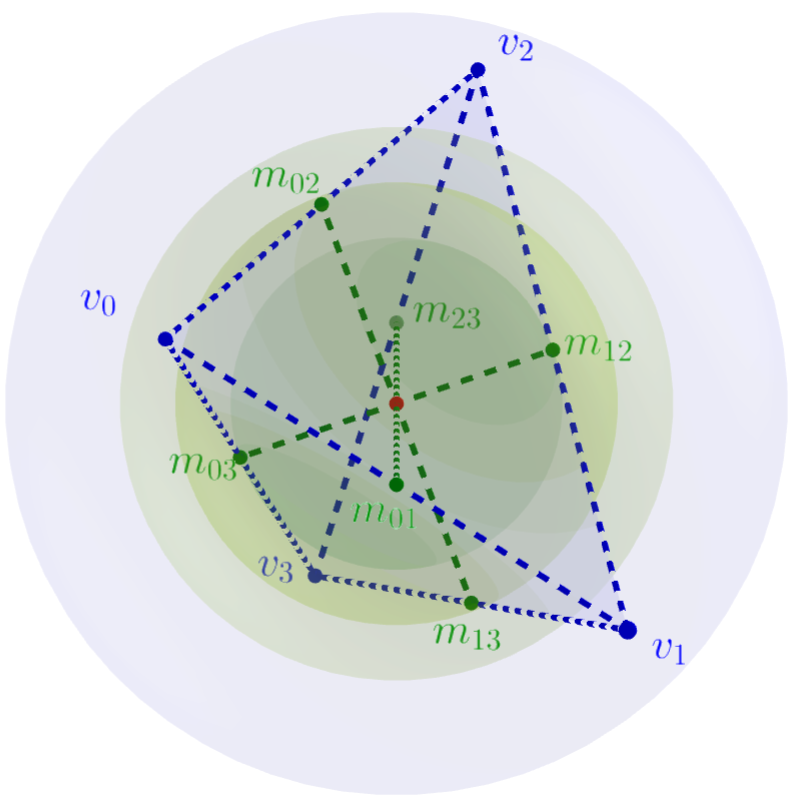}
			\includestandalone{images/tikz/radE_tanh-radH}
			\caption{In the Euclidean ball model, about the center $o$ the hyperbolic angles are conformal to the Euclidean angles, and the hyperbolic sphere $\Sph_o(r)$ is a Euclidean sphere of radius $\tanh(r)$.}
			\label{fig:re=tanh(rh)}
		\end{figure}
		
		At any $x\in \Hyp^n$, the geodesic map from its unit tangent sphere yields a homeomorphism with the boundary $\Sph_x^{n-1} \to \partial \Hyp^n$ which is $\Isom(\Hyp^n)$-equivariant: an isometry $f\colon\Hyp^n\to \Hyp^n$ extends to a projective map on the boundary $\partial f \colon \partial \Hyp^n \to \partial \Hyp^n$, and its differential $df_x\colon T_{x}\Hyp^n\to T_{f(x)}\Hyp^n$ restricts to a conformal map $df_x\colon \Sph_{o}^n\to \Sph_{f(x)}^n$: these fit into a commutative diagram.
		
		\section{Hyperbolic simplices}
		
		\subsection{Definition of simplices and related geometric constructions}
		\label{subsec:hyperbolic-simplices}
		
		\subsubsection*{Hyperbolic simplices and their skeleta}
		
		A \emph{simplex} $\Delta$ in $n$-dimensional hyperbolic space is an element of the compact $(\Hyp^n\sqcup \partial\Hyp^n)^{n+1}$, that is a sequence of $n+1$ points $v_0,\dots,v_n$ in $\Hyp^n \sqcup \partial \Hyp^n$ called its \emph{vertices}.
		The simplex is \emph{ideal} when all its vertices lie in the boundary, namely when it belongs to the compact $(\partial\Hyp^n)^{n+1}$.
		
		For $m\in \{0,\dots, n\}$, a \emph{$m$-face} of $\Delta$ is the convex hull of $1+m$ of its vertices. 
		Its \emph{$m$-skeleton} $\Delta^{(m)}$ is the union of its $m$-faces.
		The Hausdorff distance from $\Delta^{(n)}$ to the $m$-skeleton $\Delta^{(m)}$ is
		\begin{equation*}
			\delta^n_m(\Delta) = \max\{d(p, \Delta^{(m)}) \colon p \in \Delta^{(n)} \}.
		\end{equation*}
		
		The function $\delta^n_m \colon (\Hyp^n\sqcup \partial\Hyp^n)^{n+1} \to [0,\infty]$ is continuous so it admits a maximum $\mu^n_m =\max_\Delta \delta^n_m(\Delta)$ and this maximum must be achieved on a closed subset.
		We are interested in characterizing, for fixed $n>m>0$, the simplices which maximize $\mu^n_m$ and showing that this maximum is uniformly bounded.
		
		Since $\delta^n_m$ is increasing under inclusions of $n$-skeleta, and every simplex has its $n$-skeleton contained in that of an ideal simplex, the maximal loci of $\delta^n_m$ are contained in the subset of ideal simplices.
		(However, we do not yet restrict the discussion to ideal simplices.)
		
		A simplex is called \emph{total} when its vertices are independent, that is when its $n$-skeleton $\Delta^{(n)}$ has non-empty interior (or equivalently positive $n$-volume).
		The maximal locus of $\delta^n_m$ must be contained in the set of total simplices, so from now on we assume that $\Delta$ is total.
		
		\subsubsection*{Minkowski model and homogeneous coordinates}
		
		For $k\in \{0,\dots, n\}$, the $k$-th hyperface $\partial_k \Delta^{(n)}=\Conv(\{v_0,\dots,v_n\}\setminus\{v_k\})$ is co-oriented as the boundary of $\Delta^{(n)}$. 
		In the projective model $\Proj \R^{1+n}$, it spans a co-oriented hyperplane $\Span(\partial_k \Delta^{(n)})$ that is dual (by polarity with respect to the quadric $\Proj \mathbf{X}_0$) to a signed point $v_k^\star \in \Proj^+ \mathbf{X}_{-1}$.
		In particular $\Delta^{(n)}$ is the intersection of the half spaces $H_{v_k^\star}^+$.
		
		In the Minkowski model $(\R^{1+n},\langle \cdot \mid \cdot \rangle)$, there is a unique lift of the dual vertices $v_k^\star \in \mathbf{X}_{-1}$, hence a unique lift of the vertices $v_k \in \mathbf{Y}_{\ge 0}$ satisfying $\langle v_k\mid v_k^\star \rangle = 1$, so that the bases $(v_i)$ and $(v_j^\star)$ are dual.
		This defines the \emph{Minkowski model} of a simplex $\Delta$ and its dual $\Delta^\star$ in $\R^{1+n}$.
		
		The pair of dual bases $\Delta,\Delta^\star$ yields a coordinate system on $\R^{1+n}$ where $p= \sum_0^n \lambda_k v_k$ has coordinates $\lambda_k =\langle p \mid v_k^\star \rangle$ also given by the expression:
		\begin{equation*}
			\lambda_k = \frac{\det G(v_0,\dots, v_{k-1}, p, v_{k+1}, \dots, v_n)}{\det G(v_0,\dots, v_{k-1}, v_k, v_{k+1}, \dots, v_n)}.
		\end{equation*}
		This yields our homogeneous coordinate system for $p\in \Hyp^n\sqcup \partial \Hyp^n=\Proj \mathbf{Y}_{\ge 0}$.
		
		\subsubsection*{Orthogonal projections}
		
		For $p\in \Hyp^n$, its orthogonal projection on the hyperspace $H_{v_k^\star}$ can be constructed in the projective model as the intersection $p_k= [pv_k^\star)\cap \Span(\partial_k \Delta^{(n)})$, so if $p=\sum \lambda_i v_i$ then $p_k = \sum_{i\ne k} \lambda_i v_i$.
		The distance is given by $\sinh d(p,p_k)= \rvert \langle p,v_k^\star\rangle\lvert$.
		
		Beware that for some simplices, there are points $p\in \Delta^{(n)}$ whose orthogonal projection on a hyperspace $\Span(\partial_k\Delta^{(n)})$ lies outside the hyperface $\partial_k\Delta^{(n)}$, and one may construct such simplices that are total and ideal as soon as $n\ge 3$ (see Figure \ref{fig:orthogonal projections}).
		However, for any $p\in \Delta^{(n)}$, if a hyperspace $\Span(\partial_k\Delta^{(n)})$ is closest than any other, namely $\langle p\mid v_k^\star\rangle = \min_i \{ \langle p\mid v_i^\star\rangle\}$, then its orthogonal projection $p_k$ belongs to the hyperface $\partial_k\Delta^{(n)}$.
		
		For example, the orthogonal projection $h_k$ of $v_k$ on $\partial_k\Delta^{(n)}$ has homogeneous coordinates $h_k=\lambda_k v_k+\lambda_k^\star v_k^{\star}$ satisfying $0=\langle v_k^\star \mid h_k \rangle = \lambda_k - \lambda_k^\star$ thus $h_k=v_k+v_k^{\star}$.
		
		\begin{figure}[h]
			\centering
			\includegraphics[width=0.24\linewidth]{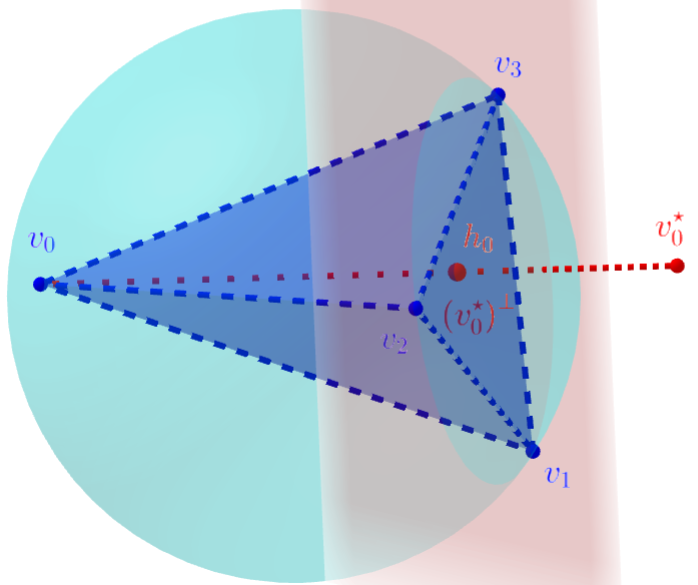}
			\hfill
			\includegraphics[width=0.32\linewidth]{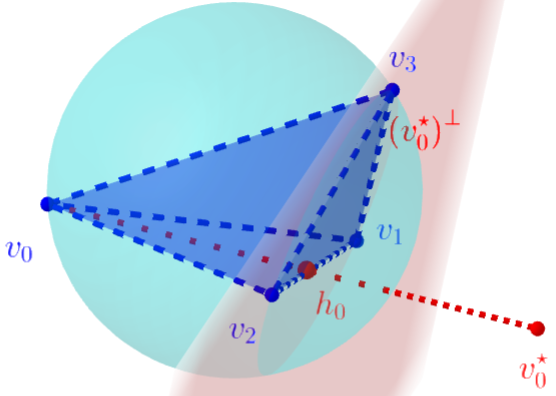}
			\hfill
			\includegraphics[width=0.32\linewidth]{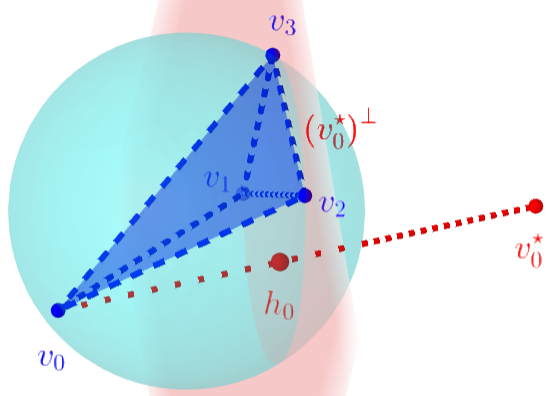}
			\caption{The orthogonal projection of $v_i\in \Delta^{(n)}$ on the hyperspace $\Span(\partial_j\Delta^{(n)})$ lies outside the hyperface $\partial_j\Delta^{(n)}$ when $\theta_{ij}^\star>\pi/2$.
				The ortholines may not intersect.}
			\label{fig:orthogonal projections}
		\end{figure}

		\subsubsection*{Gram matrices and dual gram matrices}
		
		The \emph{dual Gram matrix} of $\Delta$ is $G(\Delta^\star)=(\langle v_i^\star\mid v_j^\star \rangle)_{0\le i,j \le n}$.
		Since any two hyperfaces $\partial_i \Delta^{(n)}, \partial_j \Delta^{(n)}$ share some common vertices in $\Hyp^n\sqcup \partial \Hyp^n$, we have $\lvert \langle v_i^\star, v_j^\star \rangle \rvert \leq 1$, so the dual Gram matrix $G^\star(\Delta)$ consists of the cosines $\cos(\theta_{ij}^\star)$ of their dihedral angles denoted $\theta_{ij}^\star = \angle(\partial_i \Delta^{(n)}, \partial_j \Delta^{(n)})$.
		The \emph{Gram matrix} of $\Delta$ is $G(\Delta)=(\langle v_i \mid v_j \rangle )_{0\le i,j \le n}$.
		Recall that the vertices $v_i\in \mathbf{Y}_{>0}$ have their norms defined by the duality condition $\langle v_i\mid v_i^\star\rangle = 1$.
		
		\begin{lemma}[inverse Gram matrices]
			The Gram matrix is inverse to the dual Gram matrix and its coefficients express the basis $\Delta$ in terms of the dual basis $\Delta^\star$, in formula:
			\begin{equation}\textstyle
				\label{eq:Gram-dual-basis}
				G(\Delta)=G(\Delta^\star)^{-1}
				\qquad
				\langle v_i \mid v_j \rangle = \tfrac{(-1)^{i+j}\det G(\Delta^\star)_{\ne i, \ne j}}{\det G(\Delta^\star)}
				\qquad
				v_i=\sum_j \langle v_i \mid v_j \rangle \cdot v_j^\star
			\end{equation}
		\end{lemma}
		\begin{proof}
			Let us express the basis $\Delta$ in terms of the dual basis $\Delta^\star$: the coefficients $c_{ij}\in \R$ defined by $v_i=\sum_j c_{ij} v_j^\star$ satisfy $\langle v_i \mid v_k^\star \rangle =\sum_j c_{ij} \langle v_j^\star \mid v_k^\star \rangle$ namely $\operatorname{Id} = (c_{ij}) \times G(\Delta^\star)$, hence are the announced ratios of cofactors by determinant.
			Moreover $\langle v_i \mid v_k \rangle = \sum_{j} c_{ij} \langle v_j^\star \mid v_k \rangle = c_{ik}$.
		\end{proof}
		
		Let us mention in passing that by \cite[Theorem 7.2.4]{Ratcliffe_Foundations-Hyperbolic-Manifolds_2019}, one may characterize the minus dual Gram matrices of total simplices in $\Hyp^n$: they are all real symmetric $(1+n)\times (1+n)$ matrices $-G^\star$ with diagonal $+1$, negative determinant $\det(-G^\star)<0$, whose inverse $(-G^\star)^{-1}$ has negative entries, and whose diagonal $n\times n$ minors $(-G^\star)_{\ne k, \ne k}$ are positive definite.
		
		\subsubsection*{Regular simplices}
		
		A total simplex $\Delta$ of $\Hyp^n\sqcup \partial \Hyp^n$ is \emph{regular} when its stabilizer in $\Isom(\Hyp^n)$ acts transitively on its flags, so that it is isomorphic to the symmetric group $\mathfrak{S}_{n+1}$ of its vertices $\{v_0,\dots,v_n\}$.
		
		A regular simplex $\Delta$ must have dual Gram matrix $G(\Delta^\star)$ with diagonal $-1$ and constant off-diagonal $c\in [-1,1]$.
		We will deduce from \Cref{cor:incentred-Euclidean-model} that it exists $\iff c\in (1/n,1/(n-1)]$.
		In that case, it is unique up to the $\GO(1,n)$-action, since it is transitive on sequences in $\R^{1+n}$ with a given Gram-matrix.
		Let us compute its Gram matrix $G(\Delta)$.
		
		Denoting $\operatorname{J}_{1+n}$ the matrix with entries $1$, we may write $G(\Delta^\star)=c\operatorname{J}_{1+n} -(1+c)\operatorname{Id}_{1+n}$.
		The operator $\operatorname{J}_{1+n}$ (having eigenvalues $(1+n)$ with multiplicity $1$ and $0$ with multiplicity $n$), commutes with $\operatorname{Id}_{1+n}$ (having eigenvalues $1$ with multiplicity $1+n$).
		Thus $G(\Delta^\star)$ has eigenvalues $c(1+n)-(1+c)1=cn-1$ with multiplicity $1$ and $c0-(1+c)1$ with multiplicity $n$, so $\det(-G(\Delta^\star))=(1-cn)(1+c)^n$.
		In particular $G(\Delta)^\star$ is invertible $\iff c\notin \{-1, 1/n\}$ which we now assume, and note that $\det(-G(\Delta^\star))<0 \iff c>1/n$.
		
		Since $\operatorname{J}_{1+n}$ generates a commutative algebra of dimension $2$, the inverse belongs to this algebra so we may search $G(\Delta)=y\operatorname{J}_{1+n} - x\operatorname{Id}_{1+n}$ such that $G(\Delta)G(\Delta^\star)=\operatorname{Id}_{1+n}$ to find $x=1/(1+c)$ and $y=xc/(cn-1)$.
		For later purposes we note that $\sum_{ij} G(\Delta)_{ij} = (n+1)/(cn-1)$.
		
		In particular, $\Delta$ is ideal $\iff G(\Delta)$ has diagonal $0$ $\iff y=x \iff c=1/(n-1)$, in which case $g(\Delta)$ has off-diagonal $y=(n-1)/n$; hence $\sum_{ij} (\Delta)_{ij} = n^2-1$.
		
		\subsection{Incentred Euclidean model}
		\label{subsec:incenter}
		For a total simplex $\Delta\in (\Hyp^n\sqcup \partial \Hyp^n)^{1+n}$, the characterization of its inscribed sphere is classical and can be derived from~\cite[§6.5]{Ratcliffe_Foundations-Hyperbolic-Manifolds_2019}, but we detail a proof for completeness.
		We also provide a simple derivation of \Cref{eq:sinh-inradius-Gram} for its inradius recovering~\cite{Jacquemet_inradius-hyperbolic-simplex_2014} (where it is generalized to simplices in the extended projective space $\Proj{\R^{1+n}}$).
		Finally and most importantly, we define its \emph{inscribed dual simplex} and describe its Gram matrix in terms of the dual Gram matrix and the inradius, to discover a seemingly new \emph{inradius cosine-law} \Cref{eq:inradius-cosine-law}.
		
		\begin{proposition}[inscribed dual simplex]
			\label{prop:incenter-simplex-Hn}
			Consider a total simplex $\Delta$ of $\Hyp^n\sqcup \partial \Hyp^n$.
			
			Its \emph{incenter} $o$ is the unique point $p\in \Delta^{(n)}$ which is equidistant to all hyperplanes $\Span \partial_k \Delta^{(n)}$, and also the unique point $p\in \Delta^{(n)}$ that maximizes the distance to the $(n-1)$ skeleton $\Delta^{(n-1)}$.
			This distance $\delta^{n}_{n-1}(\Delta)$ called the \emph{inradius} is given in terms of $G(\Delta)$ and $G(\Delta^\star)$ by:
			\begin{equation}
				\label{eq:sinh-inradius-Gram}
				(\sinh \delta^n_{n-1})^{-2} 
				= \sum_{0\le i,j \le n} G(\Delta)_{ij}
				= \frac{1}{\det G(\Delta^\star)} \sum_{0\le i,j \le n} (-1)^{i+j}\det G(\Delta^\star)_{\ne i, \ne j}.
			\end{equation}
			
			The \emph{inscribed sphere} $\Sph_o(\delta^n_{n-1})$ of $\Delta$ is the unique sphere of maximal radius inscribed in $\Delta^{(n)}$: it is tangent to each hyperface $\partial_k \Delta^{(n)}$ at the orthogonal projection $v_k^\prime$ of the incenter $o$, obtained as the intersection $v_k^\prime= [ov_k^\star)\cap \partial_k \Delta^{(n)}$.
			
			We define the \emph{inscribed dual simplex} of $\Delta$ as $\Delta^\prime=(v_0^\prime,\dots,v_n^\prime) \subset (\Hyp^n)^{1+n}$.
			Its Gram matrix $G(\Delta^\prime) = (\cosh d(v_i^\prime, v_j^\prime))$ is given in terms of the dual Gram matrix $G(\Delta^\star)= (\cos \theta_{ij}^\star)$ by (\cref{fig:inradius-cosin-law_regular-simplex}): 
			\begin{equation}
				\label{eq:Gram-inscribed-dual}
				\tag{$G^\prime G^\star$}
				\langle v_i^\prime \mid v_j^\prime \rangle = 1+(\tanh \delta^n_{n-
					1})^2 (1+\langle v_i^\star \mid v_j^\star\rangle)
			\end{equation}
			implying in particular that the $\Isom(\Hyp^n)$-class of $\Delta^\prime$ determines the $\Isom(\Hyp^n)$-class of $\Delta$. 
			
			We define the \emph{incentered visual dual simplex} in the unit tangent sphere at $o$, consisting of the unit vectors $e_k^\prime \in \Sph_o(\Hyp^n)$ pointing towards $v_k'$.
			Its Euclidean \emph{incentred visual Gram matrix} $G(e_k^\prime)$ consists of their angle cosines $\langle e_i^\prime \mid e_j^\prime \rangle_o = \cos(\theta_{ij}^\prime)$.
			We find the \emph{inradius cosine law} (\cref{fig:inradius-cosin-law_regular-simplex}):
			\begin{equation}
				\tag{ICL}
				\label{eq:inradius-cosine-law}
				(\cosh \delta^n_{n-1})^2 = \frac{1+\langle v_i^\star \mid v_j^\star\rangle}{1-\langle e_i^\prime \mid e_j^\prime \rangle_o} = \frac{1+\cos(\theta_{ij})}{1-\cos(\theta_{ij}^\prime)}
			\end{equation}
			implying in particular that the $\GO(n)$-class of $(e_k^\prime)$ with $\delta^n_{n-1}$ determines the $\GO(1,n)$-class of $\Delta$.
		\end{proposition}
		
		\begin{proof}
			We will work with the Minkowski models of $\Delta,\Delta^\star$ in $(\R^{1+n},\langle \cdot \mid \cdot \rangle)$, and the associated homogeneous coordinates on $\Proj(\R^{1+n})$.
			
			On the convex set of $p\in \Delta^{(n)}$ defined by $\lvert \langle p,v_k^\star\rangle\rvert=\langle p,v_k^\star\rangle$, the equality of the $1+n$ distances $d(p,\partial_k\Delta^{(n)})$ amounts to the equality between the $1+n$ scalar products $\langle p,v_k^\star\rangle$.
			This yields a system of $n$ independent linear equations, so the solutions consist of a single line $[o]\in \Proj(\R^{1+n})$.
			Geometrically, this line is the intersection of all hyperplanes $H_{v_i-v_{j}}$ bisecting the hyperspaces $H_{v_k}$, and it is enough to consider only the intersection of the $n$ hyperplanes $H_{v_k-v_{k+1}}$.
			
			Working in the $\Delta$-homogeneous coordinates and using the duality $(\langle v_i^\star \mid v_j\rangle)_{ij} = \operatorname{Id}_{1+n}$, we find that this line passes through $\sum v_k =: so \in \mathbf{Y}_{s}$ for a unique $o\in \mathbf{Y}_{+1}$ and $s\in \R_{>0}$. 
			By definition, this point $o$ lies at the center of a sphere which is tangent to each hyperplane $\Span \partial_k \Delta^{(n)}$ at its orthogonal projection $v_k^\prime\in \partial_k \Delta^{(n)}$.
			Hence this is the unique sphere of maximal radius inscribed in $\Delta^{(n)}$, so its center $o$ is the unique point realizing the maximal distance to $\Delta^{(n-1)}$, and its radius $d(o,v_k^\prime)$ is indeed the Hausdorff distance $\delta^{n}_{n-1}$ from $\Delta^{(n)}$ to $\Delta^{(n-1)}$.
			This proves the characterizations of the incenter $o$, inradius $\delta^{n}_{n-1}(\Delta)$ and inscribed sphere.

			
			The distance $\delta^{n}_{n-1}$ from $o=\tfrac{1}{s}\sum_0^n v_k \in \mathbf{Y}_{+1}$ to the hyperplane $\Span \partial_0\Delta^{(n)}=(v_0^\star)^\perp$ satisfies $\sinh(\delta^{n}_{n-1})=\langle v_0^\star \mid o \rangle = \tfrac{1}{s} \langle v_0^\star \mid \sum_0^n v_k \rangle = 1/s$.
			Moreover we have
			\(s^2= \langle so\mid so \rangle 
			= \sum_{i,j} \langle v_i \mid v_j \rangle\) 
			and rewriting $\langle v_i\mid v_j\rangle$ using \eqref{eq:Gram-dual-basis} yields the second equality in \eqref{eq:sinh-inradius-Gram}.
			
			Since $o,v_k^\prime, v_k^\star\in \R^{1+n}$ are distinct and projectively aligned, they satisfy a linear relation of the form $v_k^\prime =\lambda_k o + \mu_k v_k^\star$.
			Expanding $1=\langle v_k^\prime \mid v_k^\prime \rangle$ yields $\lambda_k=(\cosh{\delta^n_{n-1}})^{-1}$ and expanding $0=\langle v_k^\star \mid v_k^\prime \rangle$ yields $\mu_k=(\tanh{\delta^n_{n-1})}$.  
			Hence for all $k$ we have $v_k^\prime = (\cosh{\delta^n_{n-1}})^{-1}  o+(\tanh{\delta^n_{n-1})}) v_k^\star$.
			We may thus expand $\langle v_i^\prime \mid v_j^\prime\rangle$ in terms of the inradius and $\langle v_i^\prime \mid v_j^\prime\rangle$ to find \eqref{eq:Gram-inscribed-dual}.
			
			Finally, applying the hyperbolic law of cosines to the triangle $(v_i^\prime, o, v_j^\prime)$ gives $\langle v_i^\prime \mid v_j^\prime \rangle = \cosh d(v_i^\prime, v_j^\prime) = (\cosh \delta^n_{n-1})^2-(\sinh \delta^n_{n-1})^2 (\cos \theta_{ij}^\prime)$, and combining with \eqref{eq:Gram-inscribed-dual} yields \eqref{eq:inradius-cosine-law}.
		\end{proof}
		
		\begin{figure}
			\centering
			\includegraphics[width=0.28\linewidth]{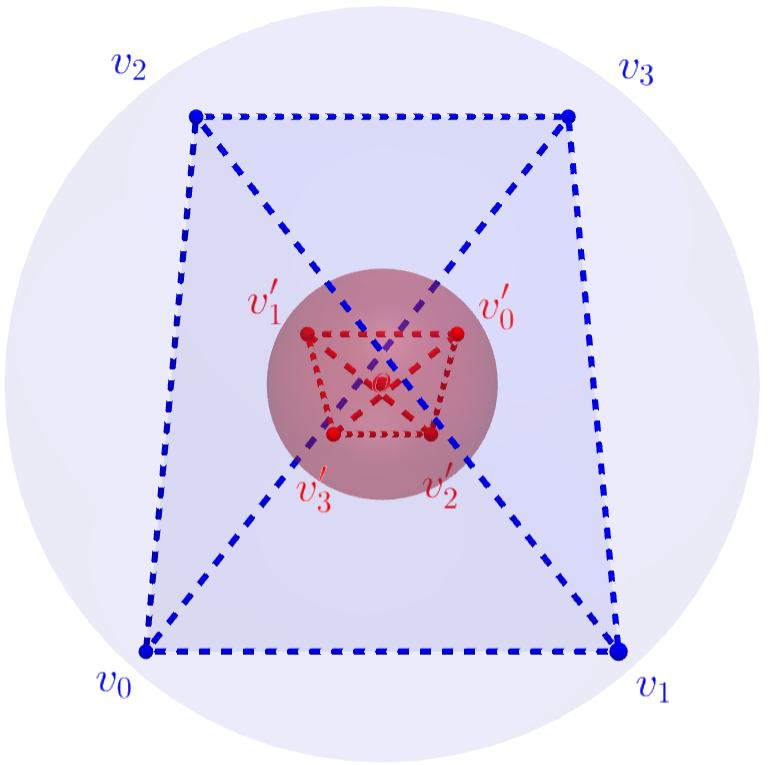}
			\qquad
			\includegraphics[width=0.5\linewidth]{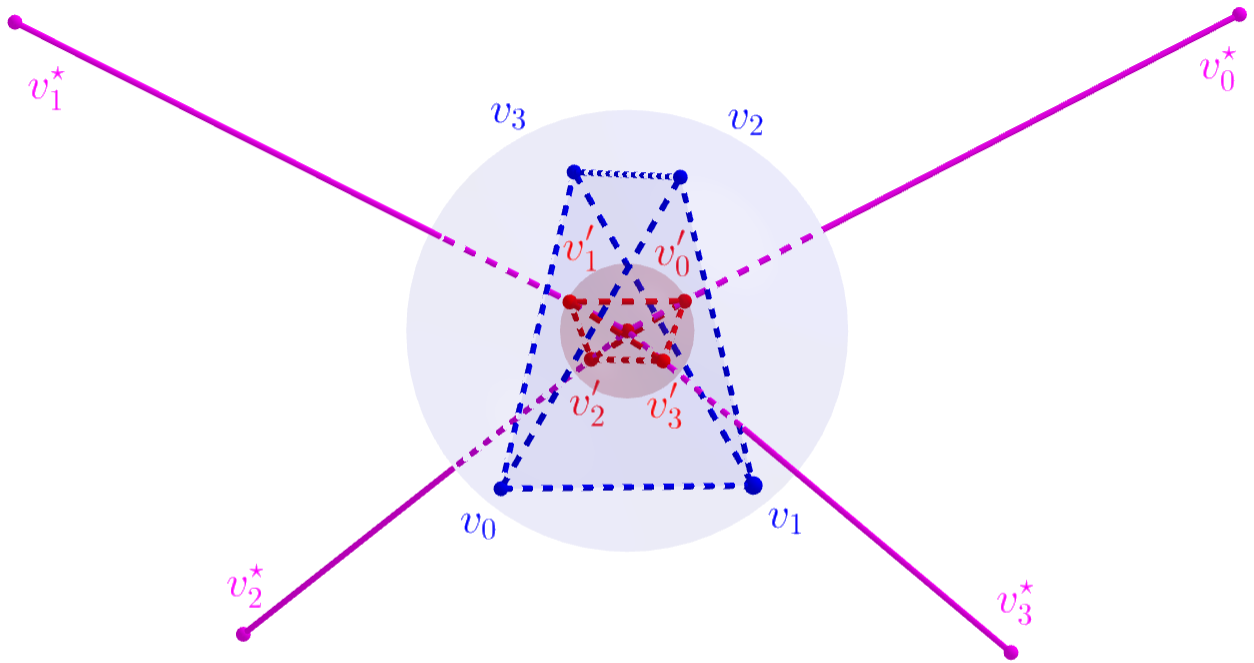}
			
			\caption{Incentred Euclidean model for $\Delta, \Delta^\prime, \Delta^\star$.}
			\label{fig:incentered-Euclidean-model}
		\end{figure}
		
		\begin{corollary}[incentred Euclidean model]
			\label{cor:incentred-Euclidean-model}
			Consider a total simplex $\Delta$ of $\Hyp^n\sqcup \partial \Hyp^n$.
			
			We may act by $\Isom(\Hyp^n)$ to place its incenter $o$ at the origin of the Euclidean chart $\R^n$ of the projective model, to define the incentered Euclidean model of $\Delta$ and $ \Delta^\prime$.
			The stabiliser of $\Delta^{(n)}$ in $\Isom(\Hyp^n)$ fixes $o$, hence equals the stabilizer of its incentred Euclidean model in $\Isom(\Sph^{n-1})$.
			
			The incentred Euclidean model of $\Delta^\prime$ is circumscribed by the sphere $\Sph_o(\tau)$ with Euclidean radius $\tau=\tanh \delta^{n}_{n-1}(\Delta)$ and has Gram matrix $\tau^2 \cdot G(e_k^\prime)=\tau^2 \cdot (\cos \theta_{ij}^\prime)$.
			
			In particular if $\Delta$ is ideal then its incentred Euclidean model is a Euclidean simplex inscribed in the unit sphere with incenter at the origin, inradius $\tanh \delta^{n}_{n-1}(\Delta)$, and dual Gram matrix $G(e_k^\prime)=(\cos \theta_{ij}^\prime)$.
			
			This yields a correspondence between $\Isom(\Hyp^n)$-classes of total ideal simplices in $\partial \Hyp^n$ and $\Isom(\Sph^{n-1})$-classes of Euclidean simplices inscribed in $\Sph^{n-1}$ with incenter at the origin.
		\end{corollary}
		
		\begin{proof}
			Placing the incenter $o$ at the origin of the Euclidean chart $\R^n=\{x\in \R^{1+n}\mid x_0=1\}$ of the projective model $\Proj{\R^{1+n}}$ yields an isometry between the tangent space at $o\in \Hyp^n$ and that Euclidean chart $\R^n$.
			In particular the tangent vectors $e_k^\prime$ at $o$ pointing towards $v_k^\prime$ and $v_k^\star$ have Euclidean angles $\theta_{ij}^\prime$.
			Moreover, the hyperbolic spheres centred at $o$ are also Euclidean spheres centred at $o$, and we know that the function sending the hyperbolic distance from $o$ to the Euclidean distance from $o$ is $r\mapsto \tanh(r)$.
			Consequently the point $v_k^\prime$, which is characterized as the tangency point of the maximally inscribed sphere in $\Delta^{(n)}$ with the hyperplane $\Span \partial_k \Delta^{(n)}$, is the orthogonal projection of $o\in  \Delta^{(n)}$ on $\Span \partial_k \Delta^{(n)}$ both for the hyperbolic and Euclidean metrics.
			The sequence of claims in the proposition follows from these observations.
		\end{proof}
		
		\begin{remark}[Minkowski to Euclid]
			The Euclidean models for $\Delta, \Delta^\star, \Delta^\prime$ in $\R^n$ are obtained by rescaling the Minkowski models for $\Delta, \Delta^\star, \Delta^\prime$ in $\R^{1+n}$ by the heights (first coordinates).
		\end{remark}
		
		\begin{remark}[regular simplices]
			\Cref{cor:incentred-Euclidean-model} implies that a total hyperbolic simplex is regular if and only if its incentred Euclidean model is regular (recovering \cite[Theorem 6.5.19]{Ratcliffe_Foundations-Hyperbolic-Manifolds_2019}), in which case it is unique up to the action of $\Isom(\Hyp^n)$.
			In particular, according to our Gram-matrix analysis in \Cref{subsec:hyperbolic-simplices}, the unique total ideal regular simplex of $\Hyp^n\sqcup \partial \Hyp^n$ has cosines of dihedral angles $\cos(\theta)=1/(n-1)$, and inradius satisfying $(\sinh \delta^n_{n-1})^{-2} = n^2-1$ of equivalently $\tanh(\delta^n_{n-1})=1/n$.
			It follows from this Gram-matrix analysis in \Cref{subsec:hyperbolic-simplices} that total regular simplices are parametrized by $\cos(\theta)\in (1/n,1/(n-1)]$ or $\tanh \delta^n_{n-1} =  \sqrt{\frac{cn-1}{n(c+1)}} \in (0, 1/n]$.
		\end{remark}
		
		\subsection{Simplices with maximal Hausdorff distance to the \texorpdfstring{$m$-skeleton}{m-skeleton}}
		
		\begin{theorem}[hyperbolic simplices of maximal inradius are regular]
			\label{thm:ideal-simplices-max-inradius}
			A total simplex $\Delta$ of $\Hyp^n\sqcup \partial \Hyp^n$ has inradius $\le \tanh^{-1}(1/n)$, with equality if and only if it is ideal and regular.
		\end{theorem}
		
		\begin{remark}[two coordinate systems]
			The Minkowski model of $\Delta$ in $\mathbf{Y}_{\ge 0}$ defines a homogeneous coordinate system for th projective space $\Proj{\R^{1+n}}$, that was used in the proof of \Cref{prop:incenter-simplex-Hn}.
			The Euclidean representative of $\Delta$ in $\Ball^n$ defines a barycentric coordinate system for the Euclidean chart $\R^n$, that will be used in the proof of \Cref{thm:ideal-simplices-max-inradius}.
			These coordinate systems are related by the heights, but we will not explicit this relation.
		\end{remark}
		
		\begin{proof}
			Since every total simplex of $\Hyp^n\sqcup \partial \Hyp^n$ is contained in an ideal one, it suffices to consider total simplices $\Delta$ of $\Hyp^n \sqcup \partial \Hyp^n$ that are ideal.
			We may work with the incentred Euclidean model: by Proposition \ref{cor:incentred-Euclidean-model} their regularity are equivalent, and their spheres centred at their common incenter $o$ coincide with their radii related by the increasing function $r\mapsto \tanh{r}$.
			Hence we must show that a Euclidean simplex $\Delta=(v_0,\dots,v_n)$ inscribed in the unit sphere $\Sph^{n-1}$ whose incenter $o$ lies at the origin has inradius at most $1/n$ with equality if and only if it is regular.
			
			In the barycentric coordinate system of $\R^n$ defined by $\Delta=(v_k) \subset \Sph_o^{n-1}$, the origin is a convex combination $o = \sum \lambda_k v_k$ of the vertices with weights $\lambda_k>0$ satisfying $\sum_0^n \lambda_k = 1$.
			The Euclidean distance from $o$ to its orthogonal projection $v_k^\prime$ on $\partial_k \Delta^{(n)}$ is at most the distance from $o$ to $\tfrac{-\lambda_k}{1-\lambda_k}v_k=\tfrac{1}{1-\lambda_k}\sum_{j\ne k} \lambda_jv_j \in \partial_k \Delta^{(n)}$, that is $\tfrac{\lambda_k}{1-\lambda_k}$, with equality if and only if $v_k^\prime=\tfrac{-\lambda_k}{1-\lambda_k}v_k$, implying that $(v_k o)$ intersects $\partial_k\Delta^{(n)}$ orthogonally at $v_k^\prime$.
			Hence the Euclidean inradius satisfies $\tanh(\delta^n_{n-1})\le \min\{\lambda_k/(1-\lambda_k) \colon 0 \le k \le n\} \le \left(\tfrac{1}{n+1}\right)/\left(1-\tfrac{1}{n+1}\right) = \tfrac{1}{n}$ with equality to $\tfrac{1}{n}$ if and only if all $k$ we have $\lambda_k=\tfrac{1}{n+1}$ and $(v_ko)$ intersects $\partial_k\Delta^{(n)}$ orthogonally at $v_k^\prime=\tfrac{-1}{n}v_k$.
			The case of equality implies that $\Delta$ admits $o$ as an orthocenter, and as $o$ coincides with the incenter and circumcenter, we deduce from \cite[Theorem 4.3]{Edmonds-Allan-Martini_Orthocentric-simplices-centers_2005} that $\Delta$ must be regular.
		\end{proof}
		
		\begin{figure}[h]
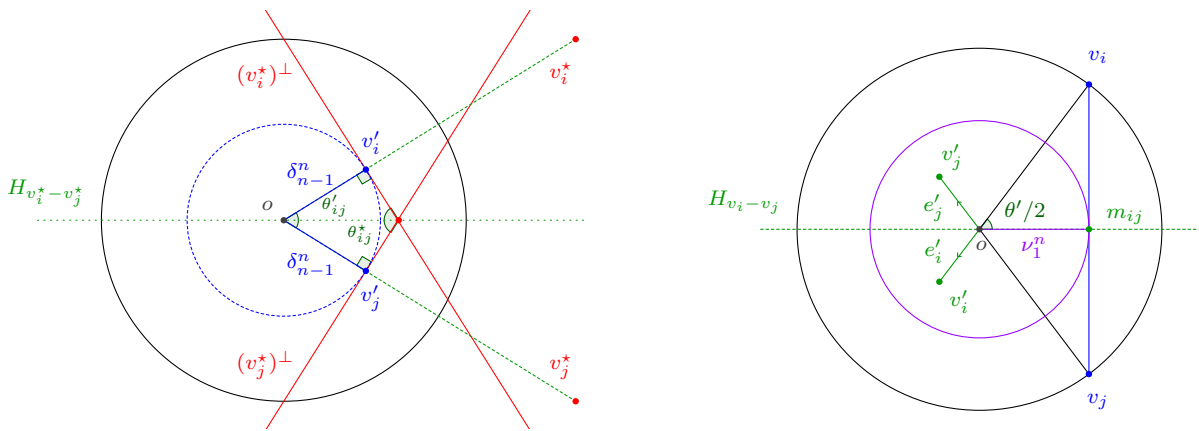

			\centering
			\includestandalone{images/tikz/inradius-cosine-law}
			\includestandalone{images/tikz/computing-theta-prime-regular}
			\caption{Inradius cosine law. Computing $\theta_{ij}'$ for a regular simplex.}
			\label{fig:inradius-cosin-law_regular-simplex}
		\end{figure}
		
		Now fix $m,n\in \N$ with $n>m>0$, and recall from \Cref{subsec:hyperbolic-simplices} that the continuous bounded function $\delta^n_m \colon (\Hyp^n\sqcup \partial\Hyp^n)^{n+1} \to \R_{\ge 0}$ achieves its maximum $\mu^n_m$ on a closed subset of total ideal simplices.
		We wish to characterize this maximizing locus.
		
		We know from \Cref{thm:ideal-simplices-max-inradius} that for $n\ge 2$, the total ideal simplices $\Delta\subset (\partial \Hyp^n)^{n+1}$ with maximal inradius $\delta^n_{n-1}(\Delta)=\mu^n_{n-1}$ are regular and $\tanh(\mu^n_{n-1}) = 1/n$.
		
		\begin{theorem}[maximal Hausdorff distance to $m$-skeleton]
			\label{thm:Hausdist-simplex-skeleta}
			For integers $n>m>0$, a simplex $\Delta$ in $\Hyp^n\sqcup \partial \Hyp^n$ satisfies $\delta^n_m(\Delta)= \mu^n_m$ if and only if $\Delta$ is total, ideal and regular.
			Moreover, $\mu^{n}_{1}$ is given 
			by \(\left(\tanh \mu^n_1\right)^2 = \tfrac{n-1}{2n}\) so $\mu^n_1$ grows as $n\to \infty$ to $\mu^\infty_1 = \tanh^{-1}(1/\sqrt{2}) = \log(1+\sqrt{2})$.
		\end{theorem}
		
		\begin{proof}
			Fix a (total ideal) simplex $\Delta$ of $\Hyp^n\sqcup \partial \Hyp^n$.
			
			For $0\le n'' \le n' \le n$ define $\delta^{n'}_{n''}(\Delta)$ to be the maximum on its $n'$-faces of their Hausdorff distance to their $n''$-skeleton:
			\(
			\delta^{n'}_{n''}(\Delta) = \max \{\delta^{n'}_{n''}(F) \colon F\subset \Delta,\, \operatorname{Card}(F)=1+n''\}
			\).
			Recall that for every point in a $n'$-face $p\in F^{(n')}$, its orthogonal projection on the closest linear span of a hyperface, say $\Span \partial_0 F^{(n')}$, actually belongs to the hyperface $\partial_0F^{(n')}$.
			Hence for $n'''\le n''\le n'$, the orthogonal projection from any point $p'\in F^{(n')}$ to a closest point $p'''\in F^{(n''')}$ factors through an orthogonal projection to a closet point $p''\in F^{(n'')}$, so by the Pythagorean theorem for hyperbolic triangles we have $\cosh d(p',F^{(n''')}) = \cosh d(p',F^{(n'')})\cosh d(p'', F^{(n''')})$.
			Taking the maximum over $p'\in F^{(n')}$ yields $\cosh \delta^{n'}_{n'''}(\Delta) \le \cosh \delta^{n'}_{n''}(\Delta)\cosh \delta^{n''}_{n'''}(\Delta)$.
			Consequently $\cosh \delta^{n}_{m}(\Delta) \le \prod_{m+1}^{n} \cosh \delta^{n'}_{n'-1}(\Delta)$ which is $\le \prod_{m+1}^{n} \cosh \mu^{n'}_{n'-1}$ with equality only if $\Delta$ is regular.
			
			For such a regular simplex $\Delta$, the supremum defining $\delta^n_m(\Delta)=\sup \{d(p,\Delta^{(m)})\colon p\in \Delta^{(n)}\}$ is achieved (only) at its incenter $o$, and equals its distance to any $m$-face.
			It could be calculated from the product formula $\cosh(\mu^n_m)=\prod_{m<n'\le n} \cosh \tanh^{-1}(1/n')$ using $\cosh \tanh^{-1}(x)=\tfrac{1}{\sqrt{1-x^2}}$ but let us provide a direct computation of $\mu^n_1$ using the incentred Euclidean model.
			
			Consider a total regular Euclidean simplex $\Delta=(v_0,\dots, v_n)$ inscribed in the unit sphere $\Sph_o^{n-1}$ and let $\nu^n_1$ be the distance from the origin to any edge $(v_iv_j)$.
			We have $\mu^n_1=\tanh^{-1}(\nu^n_1)$.
			Since the points $(v_k, o, v_k^\prime, v_k^\star)$ are aligned in this order, the unit tangent vectors $e_i^\prime, e_j^\prime$ at $o$ point in the opposite direction to the vertices $v_i,v_j$, thus denoting their common angle $\theta^\prime=\theta_{ij}^\prime$ we have $\nu^n_1=\cos(\theta^\prime/2)$ hence $(\nu^n_1)^2=\tfrac{1}{2}(1+\cos \theta^\prime)$.
			To compute $\cos(\theta^\prime)= \langle e_i^\prime \mid e_j^\prime \rangle_o$ we square the norm of $0=(\sum_0^n e_k^\prime)$ to find $0=\sum_{i,j} \langle e_i^\prime \mid e_j^\prime \rangle_o = (n+1)+2\binom{n+1}{2} \cos \theta^\prime$ hence $\cos \theta^\prime=-1/n$.
			Finally $(\tanh \mu^n_1)^2=\tfrac{1}{2}(1-1/n)$ 
			thus $\mu^n_1$ grows as $n\to \infty$ to $\mu^\infty_1 = \tanh^{-1}(1/\sqrt{2}) = \log(1+\sqrt{2})$.
		\end{proof}
		
		\begin{corollary}[Hausdorff distance convex hull to $1$-skeleton]
			\label{cor:Hausdist-polytope-1-skeleton}
			For any subset $X \subset \Hyp^n\cup \partial \Hyp^n$, every point of $\Conv(X)$ is at distance at most $\log(1+\sqrt{2})$ from $Y^{(1)}=\bigcup_{x_1,x_2 \in x}[x_1, x_2]$.
		\end{corollary}
		
		\subsection{Locating and counting points in  \texorpdfstring{$\Delta$}{Delta} realising \texorpdfstring{$\delta^n_m(\Delta)$}{deltanm(Delta)}}
		
		Recall from \Cref{cor:incentred-Euclidean-model} the equivariant correspondence between total ideal simplices in $(\partial \Hyp^n)^{n+1}$ and total Euclidean simplices inscribed in $(\Sph^{n-1})^{n+1}$ whose incenter lies at the origin.
		
		The \cite[Theorem 3.2 (ii)]{Edmonds-Allan-Martini_Coincidence-simplices-centers_2005} shows that a total simplex $\Delta$ in $\R^n$ is equicentered (its incenter $o$ coincides with its circumcenter) if and only if it is equiradial (its hyperfaces all have the same circumradius).
		In that case, the incenter $o$ projects orthogonally on each face $\partial_k \Delta^{(n)}$ to its circumcenter $v_k^\prime$.
		However $v_k^\prime$ may be distinct from the incenter of $\partial_k \Delta^{(n)}$, in which case the distance from $o$ to the $(n-2)$ skeleton will not be the maximal distance $\delta^n_{n-2}$.
		Indeed, we propose the proof of following Proposition as an exercise for the reader.
		
		\begin{proposition}[when does $d(o,\Delta^{(m)})=\delta^n_m(\Delta)$ ?]
			\label{prop:miximizers-delta-n-m}
			Let $m,n\in \N$ with $0<m<n-1$.
			Consider a total ideal simplex $\Delta \in (\partial \Hyp^n)^{n+1}$.
			
			If its incenter $o$ realizes the maximal distance to the $m$-skeleton, namely $d(o,\Delta^{(m)})= \delta^n_m(\Delta)$, then $\Delta$ is regular; thus if $\Delta$ is not regular then there are at least $2$ maximizers of $\delta^n_m$.
		\end{proposition}
		\begin{problem}[maximizers of $\delta^n_m$]
			\label{prob:maximizers_delta-n-m}
			Each $(n+1)$-set of $m$-faces $\{F_{0}^{(m)}, \dots, F_n^{(m)}\} \in \left(\genfrac{}{}{0pt}{}{\binom{\Delta}{1+m}}{1+n}\right)$
			defines a decreasing intersection of convex sets in $\Delta^{(n)}$ by $C_t(F)=\{p\in \Delta^{(n)} \colon d(p,F_i^{(m)}\ge t\}$, hence a point maximizing the distance to these faces, that is a local maximizer of the Hausdorff distance from $\Delta^{(n)}$ to $\Delta^{(m)}$, which for some $F$ will be a global maximer of $\delta^n_m$.
			
			\emph{What can be the number of distinct local global minimizers and where can they be located?}
			
			Note that the isometry group of the simplex $\operatorname{Stab}(\Delta^{(n)})\subset \Isom(\Hyp^n)$ acts on that set of $(1+n)$-sets of $m$-faces.
			However the description and enumeration of distinct (local or global) maximizers is more subtle than the question of classifying (all or certain) of the orbits under this action, since it may happen that two $(1+n)$-sets of $m$-faces yield the same (local or global) maximizer without belonging to the same orbit.
		\end{problem}
		
		\begin{example}[disphenoids]
			\label{eg:disphenoids}
			The tetrahedra in $\R^3$ that are equicentred or equiradial are called \emph{disphenoids}, and have various equivalent characterisations (see \cite{Lemoine_tetraedres-equifacetaux_1880}), such as having opposite edges of the same length and acute triangle faces.
			This shows that acute triangles parametrize disphenoids by folding along their midlines (joining the mid-edges) as in Figure \ref{fig:disphenoid}, and this parametrization is equivariant under the actions of the groups of Euclidean similitudes.
			\begin{figure}[h]
				\centering
				\includegraphics[width=0.324\linewidth]{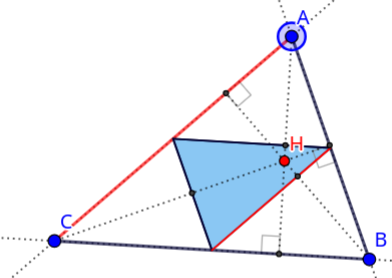}
				\includegraphics[width=0.324\linewidth]{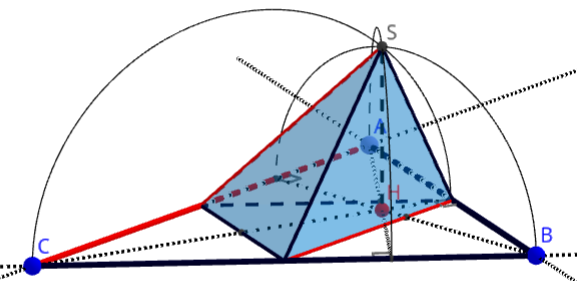}
				\includestandalone{images/tikz/intersection}
				\caption{From acute triangles to disphenoids. Maximizers of $\delta^3_1$. Moduli space of disphenoids.}
				\label{fig:disphenoid}
			\end{figure}
			
			Identifying the Euclidean plane $\R^2$ with the complex line $\C$, every acute triangle has a unique representative modulo the action of the group of similitudes $\C\rtimes \C^\times = \R^2 \rtimes (\operatorname{SO}(2)\times \R_{>0}^*)$, of the form $(1,z,-1)$ where $z\in \C$ has $\Im(z)>0$ and $-1<\Re(z)<1$ and $\lvert z \rvert >1$.
			(One may compare this parametrization with the cross-ratio of the four vertices in its circumscribed sphere.)
			The (pairs of opposite) edges of the associated folded tetrahedron have lengths $\lvert z+1\rvert, 2, \lvert z-1\vert$.
			The number of (local, global) maximizers of $\delta^3_1$ is $(1,1)$ when $\Im(z)=0$ and $\lvert z-1\rvert = 1$ namely the faces are equilateral, or $(4,4)$ when $\lvert z \rvert >1$ but $\Re(z)=0$ namely the faces are strictly isoceles, and $(4,1)$ when the faces are generic triangles.
			\begin{figure}[h]
				\centering
				\includegraphics[width=0.31\linewidth]{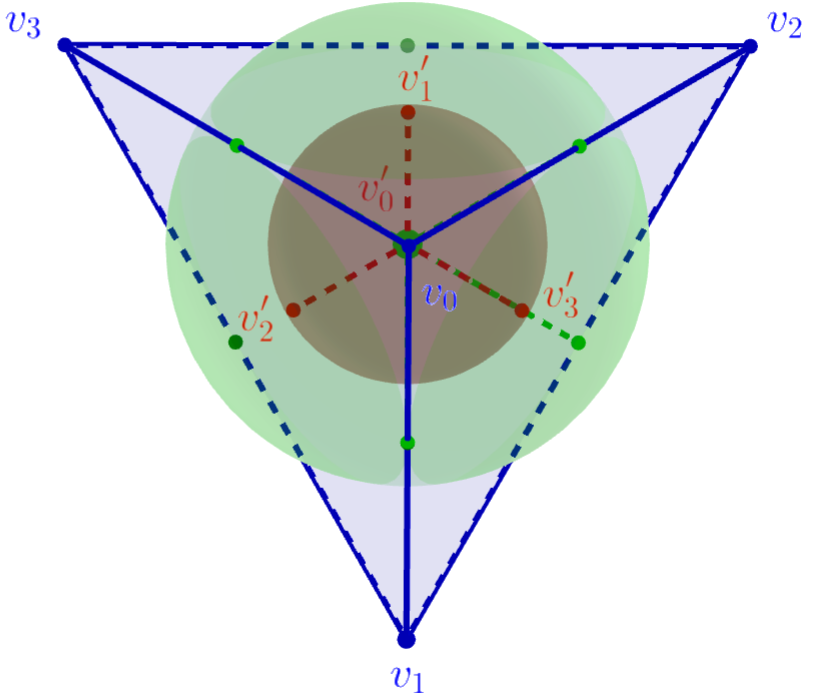}
				\includegraphics[width=0.31\linewidth]{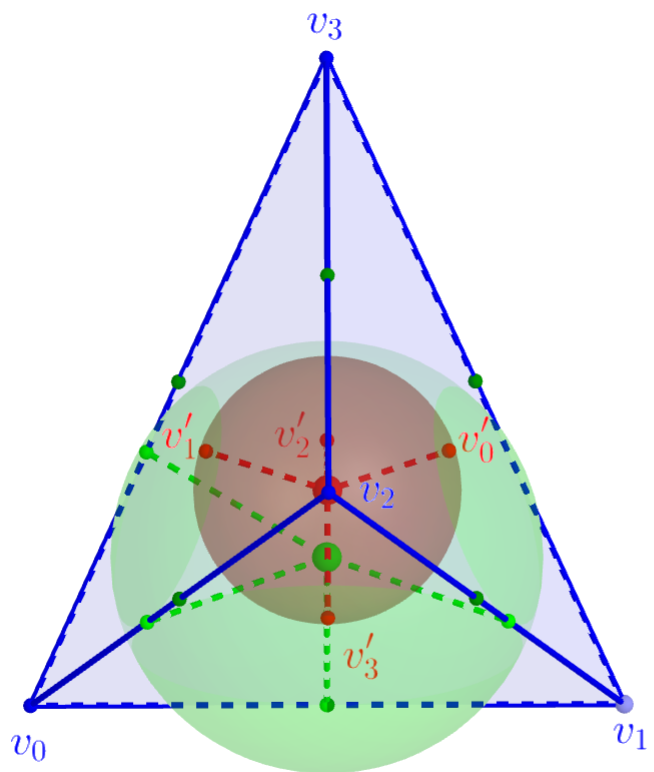}
				\includegraphics[width=0.324\linewidth]{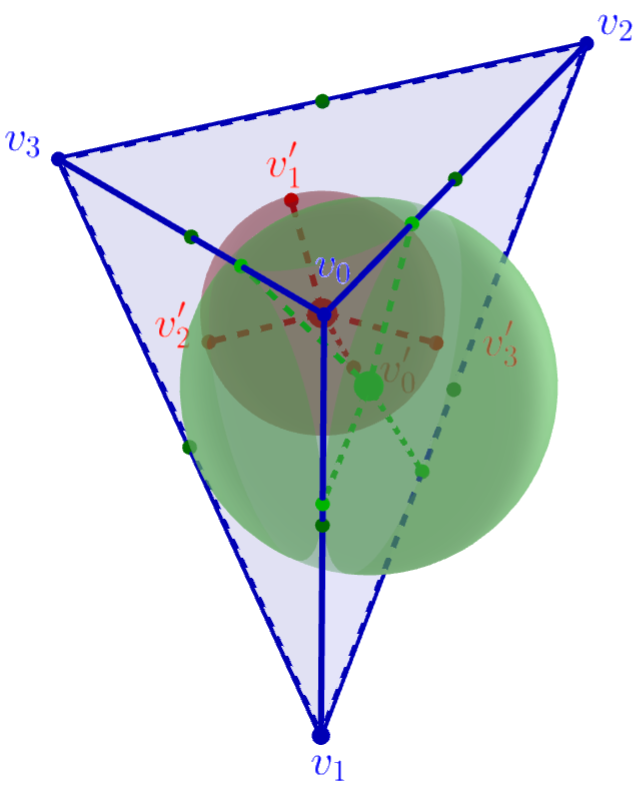}
				\caption{Maximizers of $\delta^3_1$ in the three cases.}
				\label{fig:maximizer-delta-31}
			\end{figure}
		\end{example}
		
		\section*{Acknowledgements}
		
		We express our thanks to Anton Petrunin for advising the use homogneous coordinates in the proof of \Cref{thm:ideal-simplices-max-inradius}.
		
		\bibliographystyle{alpha}
		\bibliography{biblio.bib}
		
	\end{document}

%% file: images/tikz/config-hyperplanes-positive.tex
\definecolor{qqwuqq}{rgb}{0,0.39,0}
\definecolor{qqzzqq}{rgb}{0,0.6,0}
\definecolor{qqqqff}{rgb}{0,0,1}
\definecolor{ffqqqq}{rgb}{1,0,0}
\begin{tikzpicture}[line cap=round,line join=round,>=triangle 45,x=1.0cm,y=1.0cm,scale=1.55]
\clip(-1.7,-1.1) rectangle (1.7,1.4);
\draw [shift={(0.63,0)},color=qqwuqq,fill=qqwuqq,fill opacity=0.1] (0,0) -- (116.57:0.2) arc (116.57:243.43:0.2) -- cycle;
\draw(0,0) circle (1cm);
\draw [color=qqqqff,domain=-1.7:1.7] plot(\x,{(-1.48--0.19*\x)/-1.47});
\draw [color=qqqqff,domain=-1.7:1.7] plot(\x,{(--1.48-1.29*\x)/-0.73});
\draw [color=ffqqqq,domain=-1.7:1.7] plot(\x,{(--1-1.6*\x)/0.8});
\draw [color=qqqqff,domain=-1.7:1.7] plot(\x,{(-1.48--1.29*\x)/-0.73});
\draw [color=qqqqff,domain=-1.7:1.7] plot(\x,{(--1.48-0.19*\x)/-1.47});
\draw [color=ffqqqq,domain=-1.7:1.7] plot(\x,{(--1-1.6*\x)/-0.8});
\draw [dash pattern=on 1pt off 1pt,color=qqzzqq,domain=-1.7:1.7] plot(\x,{(-0-0*\x)/-6.59});
\draw [dash pattern=on 1pt off 1pt,color=qqzzqq] (0.63,-1.1) -- (0.63,1.1);
\begin{scriptsize}
\fill [color=ffqqqq] (1.6,0.8) circle (1.0pt);
\draw[color=ffqqqq] (1.6,0.5) node {$u$};
\draw[color=ffqqqq] (0,.8) node {$u^\perp$};
\fill [color=ffqqqq] (1.6,-0.8) circle (1.0pt);
\draw[color=ffqqqq] (1.6,-.5) node {$v$};
\draw[color=ffqqqq] (1.35,1.05) node {$v^\perp$};
\fill [color=qqzzqq] (1.15,0) circle (1.0pt);
\fill [color=ffqqqq] (0.63,0) circle (1.0pt);
\fill [color=qqzzqq] (7.74,0) circle (1.0pt);
\draw[color=qqzzqq] (-1.4,0.2) node[font=\scriptsize] {$H_{u-v}$};
\fill [color=qqzzqq] (0.63,0.93) circle (0.5pt);
\fill [color=qqzzqq] (0.63,-0.93) circle (0.5pt);
\draw[color=qqzzqq] (0.8,1.2) node[font=\scriptsize] {$H_{u+v}$};
\draw[color=qqwuqq] (-.1,0.15) node[font=\scriptsize] {$0 < \theta < \pi/2$};
\end{scriptsize}
\end{tikzpicture}

%% file: images/tikz/config-hyperplanes-null.tex
\definecolor{qqwuqq}{rgb}{0,0.39,0}
\definecolor{qqzzqq}{rgb}{0,0.6,0}
\definecolor{qqqqff}{rgb}{0,0,1}
\definecolor{ffqqqq}{rgb}{1,0,0}
\begin{tikzpicture}[line cap=round,line join=round,>=triangle 45,x=1.0cm,y=1.0cm,scale=1.55]
\clip(-1.7,-1.2) rectangle (1.4,1.4);
\draw [shift={(1,0)},color=qqwuqq,fill=qqwuqq,fill opacity=0.1] (0,0) -- (135:0.2) arc (135:225:0.2) -- cycle;
\draw(0,0) circle (1cm);
\draw [color=qqqqff,domain=-1.5:1.4] plot(\x,{(-1-0*\x)/-1});
\draw [color=qqqqff] (1,-1.2) -- (1,1.2);
\draw [color=ffqqqq,domain=-1.5:1.4] plot(\x,{(--1-1*\x)/1});
\draw [color=qqqqff] (1,-1.2) -- (1,1.2);
\draw [color=qqqqff,domain=-1.5:1.4] plot(\x,{(--1-0*\x)/-1});
\draw [color=ffqqqq,domain=-1.5:1.4] plot(\x,{(--1-1*\x)/-1});
\draw [dash pattern=on 1pt off 1pt,color=qqzzqq] (1,-1.2) -- (1,1.2);
\draw [dash pattern=on 1pt off 1pt,color=qqzzqq,domain=-1.5:1.4] plot(\x,{(-0-0*\x)/2});
\begin{scriptsize}
\fill [color=ffqqqq] (1,1) circle (1.0pt);
\draw[color=ffqqqq] (1.1,0.85) node {$u$};
\draw[color=ffqqqq] (0.05,0.7) node {$u^\perp$};
\fill [color=ffqqqq] (1,-1) circle (1.0pt);
\draw[color=ffqqqq] (1.1,-0.85) node {$v$};
\draw[color=ffqqqq] (0.05,-0.7) node {$v^\perp$};
\fill [color=ffqqqq] (1,0) circle (1.0pt);
\fill [color=qqzzqq] (1,1) circle (1.0pt);
\fill [color=qqzzqq] (1,-1) circle (1.0pt);
\draw[color=qqzzqq] (0.65,1.2) node[font=\scriptsize] {$H_{u+v}$};
\draw[color=qqzzqq] (-1.35,0.15) node[font=\scriptsize] {$H_{u-v}$};
\draw[color=qqwuqq] (0.4,0.15) node {$\theta=0$};
\end{scriptsize}
\end{tikzpicture}

%% file: images/tikz/config-hyperplanes-negative.tex
\definecolor{qqwuqq}{rgb}{0,0.39,0}
\definecolor{qqzzqq}{rgb}{0,0.6,0}
\definecolor{qqqqff}{rgb}{0,0,1}
\definecolor{ffqqqq}{rgb}{1,0,0}
\begin{tikzpicture}[line cap=round,line join=round,>=triangle 45,x=1.0cm,y=1.0cm,scale=1.55]
\clip(-1.5,-1.3) rectangle (1.5,1.8);
\draw [shift={(-0.69,-0.2)},color=qqwuqq,fill=qqwuqq,fill opacity=0.1] (0,0) -- (0:0.17) arc (0:98.13:0.17) -- cycle;
\draw [shift={(0.69,-0.2)},color=qqwuqq,fill=qqwuqq,fill opacity=0.1] (0,0) -- (81.87:0.17) arc (81.87:180:0.17) -- cycle;
\draw(0,0) circle (1cm);
\draw [color=qqqqff,domain=-1.5:1.5] plot(\x,{(-1-0.6*\x)/0.8});
\draw [color=qqqqff,domain=-1.5:1.5] plot(\x,{(--1--0.8*\x)/0.6});
\draw [color=ffqqqq,domain=-1.5:1.5] plot(\x,{(--1--1.4*\x)/-0.2});
\draw [color=qqqqff,domain=-1.5:1.5] plot(\x,{(-1--0.8*\x)/-0.6});
\draw [color=qqqqff,domain=-1.5:1.5] plot(\x,{(--1-0.6*\x)/-0.8});
\draw [color=ffqqqq,domain=-1.5:1.5] plot(\x,{(--1-1.4*\x)/-0.2});
\draw [dash pattern=on 1pt off 1pt,color=qqzzqq] (0,-1.3) -- (0,1.8);
\draw [dash pattern=on 1pt off 1pt,color=qqzzqq,domain=-1.5:1.5] plot(\x,{(--0.56-0*\x)/-2.8});
\draw [color=qqzzqq] (-0.69,-0.2)-- (0.69,-0.2);
\begin{scriptsize}
\fill [color=ffqqqq] (-1.4,-0.2) circle (1.0pt);
\draw[color=ffqqqq] (-1.4,0) node {$u$};
\fill [color=ffqqqq] (1.4,-0.2) circle (1.0pt);
\draw[color=ffqqqq] (1.39,0) node {$v$};
\draw[color=ffqqqq] (1.1,1) node {$u^\perp$};
\draw[color=ffqqqq] (-1.1,1) node {$v^\perp$};
\fill [color=qqzzqq] (0,1.67) circle (1.0pt);
\fill [color=qqzzqq] (0,-1.25) circle (1.0pt);
\draw[color=qqzzqq] (0.45,1.61) node[font=\scriptsize] {$H_{u-v}$};
\fill [color=ffqqqq] (-0.69,-0.2) circle (1.0pt);
\fill [color=ffqqqq] (0.69,-0.2) circle (1.0pt);
\fill [color=qqzzqq] (0,-0.2) circle (1.0pt);
\draw[color=qqwuqq] (-0.49,0.1) node {$\pi/2$};
\draw[color=qqwuqq] (0.44,0.1) node {$\pi/2$};
\end{scriptsize}
\end{tikzpicture}

%% file: images/tikz/distance-from-cross-ratio.tex
\definecolor{qqqqff}{rgb}{0,0,1}
\definecolor{ffqqqq}{rgb}{1,0,0}
\begin{tikzpicture}[line cap=round,line join=round,>=triangle 45,x=1.0cm,y=1.0cm,scale=.7]
\clip(-3.1,-3.1) rectangle (3.1,3.1);
\draw(0,0) circle (3cm);
\draw [dash pattern=on 2pt off 2pt,color=qqqqff,domain=-3.1:3.1] plot(\x,{(-0-0*\x)/2});
\draw (-2.5,-0.3) node[anchor=north west,font=\scriptsize] {$d(x,y)=\frac{1}{2}\log  | [x',x,y,y'] |$};
\draw [color=ffqqqq] (-1,0)-- (1,0);
\begin{scriptsize}
\fill [color=ffqqqq] (-1,0) circle (1.5pt);
\draw[color=ffqqqq] (-0.95,0.35) node {$x$};
\fill [color=ffqqqq] (1,0) circle (1.5pt);
\draw[color=ffqqqq] (1.05,0.35) node {$y$};
\fill [color=qqqqff] (3,0) circle (1.5pt);
\draw[color=qqqqff] (2.59,0.35) node {$y'$};
\fill [color=qqqqff] (-3,0) circle (1.5pt);
\draw[color=qqqqff] (-2.52,0.35) node {$x'$};
\end{scriptsize}
\end{tikzpicture}

%% file: images/tikz/polar-from-incidence.tex
\definecolor{qqzzqq}{rgb}{0,0.6,0}
\definecolor{qqqqff}{rgb}{0,0,1}
\definecolor{xdxdff}{rgb}{0.49,0.49,1}
\definecolor{ffqqqq}{rgb}{1,0,0}
\begin{tikzpicture}[line cap=round,line join=round,>=triangle 45,x=1.0cm,y=1.0cm,scale=1.5]
\clip(-1.1,-1.65) rectangle (2.1,1.1);
\draw(0,0) circle (1cm);
\draw [dash pattern=on 1pt off 1pt,color=qqqqff,domain=-1.1:2.1] plot(\x,{(-1.08--0.6*\x)/-2.6});
\draw [dash pattern=on 1pt off 1pt,color=qqqqff,domain=-1.1:2.1] plot(\x,{(--1.77-0.98*\x)/-1.6});
\draw [dash pattern=on 1pt off 1pt,color=qqzzqq,domain=-1.1:2.1] plot(\x,{(-0.67-1.58*\x)/1});
\draw [dash pattern=on 1pt off 1pt,color=qqzzqq,domain=-1.1:2.1] plot(\x,{(--0.76-0.81*\x)/-0.2});
\draw [dash pattern=on 1pt off 1pt,color=qqzzqq,domain=-1.1:2.1] plot(\x,{(-0.03-1.22*\x)/1.58});
\draw [dash pattern=on 1pt off 1pt,color=qqzzqq,domain=-1.1:2.1] plot(\x,{(--1-1.17*\x)/-0.79});
\draw [color=ffqqqq] (0.56,-1.65) -- (0.56,1.1);
\begin{scriptsize}
\fill [color=ffqqqq] (1.8,0) circle (1.0pt);
\draw[color=ffqqqq] (1.79,0.14) node {$u$};
\fill [color=xdxdff] (-0.8,0.6) circle (0.5pt);
\fill [color=xdxdff] (0.2,-0.98) circle (0.5pt);
\fill [color=qqqqff] (0.2,-0.98) circle (1.0pt);
\fill [color=qqqqff] (0.78,-0.62) circle (1.0pt);
\fill [color=qqqqff] (-0.8,0.6) circle (1.0pt);
\fill [color=qqqqff] (0.98,0.19) circle (1.0pt);
\fill [color=qqzzqq] (0.56,-0.45) circle (1.0pt);
\fill [color=qqzzqq] (0.56,-1.55) circle (1.0pt);
\draw[color=ffqqqq] (0.38,.72) node {$u^\perp$};
\end{scriptsize}
\end{tikzpicture}

%% file: images/tikz/orthoproj-bissectors.tex
\definecolor{qqwuqq}{rgb}{0,0.39,0}
\definecolor{qqqqff}{rgb}{0,0,1}
\definecolor{qqzzqq}{rgb}{0,0.6,0}
\definecolor{ffqqqq}{rgb}{1,0,0}
\definecolor{uququq}{rgb}{0.25,0.25,0.25}
\begin{tikzpicture}[line cap=round,line join=round,>=triangle 45,x=1.0cm,y=1.0cm,scale=1.5,font=\scriptsize]
\clip(-2.07,-1.5) rectangle (1.37,1.49);
\draw[color=qqwuqq,fill=qqwuqq,fill opacity=0.1] (0.77,0.38) -- (0.83,0.42) -- (0.8,0.48) -- (0.74,0.44) -- cycle; 
\draw[color=qqwuqq,fill=qqwuqq,fill opacity=0.1] (0.38,-0.77) -- (0.42,-0.83) -- (0.48,-0.8) -- (0.44,-0.74) -- cycle; 
\draw[color=qqwuqq,fill=qqwuqq,fill opacity=0.1] (-0.7,0.5) -- (-0.76,0.54) -- (-0.8,0.48) -- (-0.74,0.44) -- cycle; 
\draw(0,0) circle (1cm);
\draw [color=ffqqqq,domain=-2.07:1.37] plot(\x,{(--1-0.6*\x)/-1});
\draw [color=ffqqqq,domain=-2.07:1.37] plot(\x,{(--1-1*\x)/0.6});
\draw [dash pattern=on 1pt off 1pt,color=qqzzqq,domain=-2.07:1.37] plot(\x,{(-0-0.75*\x)/3});
\draw [color=qqqqff] (0,0)-- (1,0.6);
\draw [color=qqqqff] (0,0)-- (0.6,-1);
\draw [color=ffqqqq,domain=-2.07:1.37] plot(\x,{(--1--1*\x)/0.6});
\draw [dash pattern=on 1pt off 1pt,color=qqzzqq,domain=-2.07:1.37] plot(\x,{(-0--3.43*\x)/3.43});
\draw [color=qqqqff] (0,0)-- (-1,0.6);
\draw [dash pattern=on 1pt off 1pt,color=qqzzqq] (0,-1.37) -- (0,1.49);
\draw [color=qqqqff] (0,0) circle (0.86cm);
\begin{scriptsize}
\fill [color=uququq] (0,0) circle (1.0pt);
\draw[color=uququq] (0.1,0.2) node {$o$};
\fill [color=ffqqqq] (0.6,-1) circle (1.0pt);
\draw[color=ffqqqq] (0.74,-1.03) node {$u$};
\draw[color=ffqqqq] (-0.24,-1.3) node {$u^\perp$};
\fill [color=ffqqqq] (1,0.6) circle (1.0pt);
\draw[color=ffqqqq] (1.12,0.61) node {$v$};
\draw[color=ffqqqq] (0.39,1.39) node {$v^\perp$};
\draw[color=qqzzqq] (-1.6,0.6) node[font=\tiny] {$H_{u-v}$};
\fill [color=ffqqqq] (-1,0.6) circle (1.0pt);
\draw[color=ffqqqq] (-1.15,0.5) node {$w$};
\draw[color=ffqqqq] (-0.4,1.37) node {$w^\perp$};
\draw[color=qqzzqq] (.84,1.25) node[font=\tiny] {$H_{u-w}$};
\fill [color=qqqqff] (0.44,-0.74) circle (1.0pt);
\draw[color=qqqqff] (0.47,-0.5) node {$o_u$};
\fill [color=qqqqff] (0.74,0.44) circle (1.0pt);
\draw[color=qqqqff] (0.66,0.26) node {$o_v$};
\fill [color=qqqqff] (-0.74,0.44) circle (1.0pt);
\draw[color=qqqqff] (-0.47,0.45) node {$o_w$};
\draw[color=qqzzqq] (.3,-1.25) node[font=\tiny] {$H_{v-w}$};
\end{scriptsize}
\end{tikzpicture}

%% file: images/tikz/radE_tanh-radH.tex
\definecolor{qqzzqq}{rgb}{0,0.6,0}
\definecolor{ffqqqq}{rgb}{1,0,0}
\definecolor{qqqqff}{rgb}{0,0,1}
\begin{tikzpicture}[line cap=round,line join=round,>=triangle 45,x=1.0cm,y=1.0cm,scale=2]
\draw[->,color=black] (-1.14,0) -- (1.46,0);
\foreach \x in {-1,1}
\draw[shift={(\x,0)},color=black] (0pt,2pt) -- (0pt,-2pt);
\draw[->,color=black] (0,-0.15) -- (0,1.78);
\foreach \y in {,1}
\draw[shift={(0,\y)},color=black] (2pt,0pt) -- (-2pt,0pt);
\clip(-1.14,-0.15) rectangle (1.7,1.78);
\draw [color=qqqqff,domain=-1.14:1.46] plot(\x,{(-0--1*\x)/1});
\draw [color=qqqqff,domain=-1.14:1.46] plot(\x,{(-0-1*\x)/1});

\draw [samples=200,domain=-1.5:1.5,color=ffqqqq] plot (\x,{sqrt(1+(\x)^2)});
\draw (0,0)-- (1.16,1.53);
\draw [dash pattern=on 1pt off 1pt,color=qqzzqq] (-1,1)-- (1,1);
\draw [line width=1.2pt,dotted,color=ffqqqq] (0,1.53)-- (1.16,1.53);
\draw [line width=1.2pt,dotted,color=ffqqqq] (1.16,0)-- (1.16,1.53);
\draw [line width=1.2pt,color=qqzzqq] (0,1)-- (0.76,1);
\begin{scriptsize}
\fill [color=qqqqff] (0,0) circle (.5pt);
\fill [color=ffqqqq] (1.16,1.53) circle (.5pt);
\draw[color=ffqqqq] (1.15,1.65) node {$q$};
\fill [color=qqzzqq] (0.76,1) circle (.5pt);
\draw[color=qqzzqq] (0.72,1.1) node {$p$};
\fill [color=qqqqff] (-1,1) circle (.5pt);
\fill [color=qqqqff] (1,1) circle (.5pt);
\fill [color=qqzzqq] (0,1) circle (.5pt);
\draw[color=qqzzqq] (0.08,1.08) node {$o$};
\draw[color=ffqqqq,font=\scriptsize] (0.56,1.7) node {$\sinh(r)$};
\draw[color=ffqqqq,font=\scriptsize] (1.45,0.79) node {$\cosh(r)$};
\draw[color=qqzzqq] (0.3,0.8) node {$\frac{\sinh(r)}{\cosh(r)}$};
\end{scriptsize}
\end{tikzpicture}

%% file: images/tikz/inradius-cosine-law.tex
\definecolor{qqwuqq}{rgb}{0,0.39,0}
\definecolor{qqqqff}{rgb}{0,0,1}
\definecolor{qqzzqq}{rgb}{0,0.6,0}
\definecolor{ffqqqq}{rgb}{1,0,0}
\definecolor{uququq}{rgb}{0.25,0.25,0.25}
\begin{tikzpicture}[line cap=round,line join=round,>=triangle 45,x=1.0cm,y=1.0cm,scale=2.4,font=\scriptsize]
\clip(-1.6,-1.15) rectangle (1.67,1.16);
\draw [shift={(0,0)},color=qqwuqq,fill=qqwuqq,fill opacity=0.1] (0,0) -- (-32.01:0.08) arc (-32.01:32.01:0.08) -- cycle;
\draw [shift={(0.63,0)},color=qqwuqq,fill=qqwuqq,fill opacity=0.1] (0,0) -- (122.01:0.08) arc (122.01:237.99:0.08) -- cycle;
\draw[color=qqwuqq,fill=qqwuqq,fill opacity=0.1] (0.4,0.25) -- (0.43,0.2) -- (0.48,0.23) -- (0.45,0.28) -- cycle; 
\draw[color=qqwuqq,fill=qqwuqq,fill opacity=0.1] (0.48,-0.23) -- (0.43,-0.2) -- (0.4,-0.25) -- (0.45,-0.28) -- cycle; 
\draw(0,0) circle (1cm);
\draw [color=ffqqqq,domain=-1.35:1.67] plot(\x,{(--1-1.6*\x)/-1});
\draw [color=ffqqqq,domain=-1.35:1.67] plot(\x,{(--1-1.6*\x)/1});
\draw [dotted,color=qqzzqq,domain=-1.35:1.67] plot(\x,{(-0-0*\x)/-5.12});
\draw [dash pattern=on 1pt off 1pt,color=qqzzqq] (0,0)-- (1.6,1);
\draw [dash pattern=on 1pt off 1pt,color=qqzzqq] (0,0)-- (1.6,-1);
\draw [dash pattern=on 1pt off 1pt,color=qqqqff] (0,0) circle (0.53cm);
\draw [color=qqqqff] (0,0)-- (0.45,0.28);
\draw [color=qqqqff] (0,0)-- (0.45,-0.28);
\fill [color=uququq] (0,0) circle (.5pt);
\draw[color=uququq] (-0.09,0.08) node {$o$};
\fill [color=ffqqqq] (1.6,-1) circle (.5pt);
\draw[color=ffqqqq] (1.52,-0.81) node {$v_j^\star$};
\draw[color=ffqqqq] (-0.1,-.8) node {$(v_j^\star)^\perp$};
\fill [color=ffqqqq] (1.6,1) circle (.5pt);
\draw[color=ffqqqq] (1.52,0.81) node {$v_i^\star$};
\draw[color=ffqqqq] (-0.1,.8) node {$(v_i^\star)^\perp$};
\draw[color=qqzzqq] (-1.3,0.15) node[font=\scriptsize] {$H_{v_i^\star-v_j^\star}$};
\fill [color=qqqqff] (0.45,-0.28) circle (.5pt);
\draw[color=qqqqff] (0.48,-0.43) node {$v_j^\prime$};
\fill [color=qqqqff] (0.45,0.28) circle (.5pt);
\draw[color=qqqqff] (0.48,0.43) node {$v_i^\prime$};
\draw[color=qqwuqq] (0.28,.08) node[font=\tiny] {$\theta_{ij}^\prime$};
\fill [color=ffqqqq] (0.63,0) circle (.5pt);
\draw[color=qqwuqq] (0.42,-0.08) node[font=\tiny] {$\theta_{ij}^\star$};
\draw[color=qqqqff] (0.15,0.25) node {$\delta^n_{n-1}$};
\draw[color=qqqqff] (0.15,-0.25) node {$\delta^n_{n-1}$};
\end{tikzpicture}

%% file: images/tikz/computing-theta-prime-regular.tex
\definecolor{qqwuqq}{rgb}{0,0.39,0}
\definecolor{zzqqff}{rgb}{0.6,0,1}
\definecolor{qqzzqq}{rgb}{0,0.6,0}
\definecolor{qqqqff}{rgb}{0,0,1}
\definecolor{uququq}{rgb}{0.25,0.25,0.25}
\begin{tikzpicture}[line cap=round,line join=round,>=triangle 45,x=1.0cm,y=1.0cm,scale=2.4]
\clip(-1.8,-1.1) rectangle (1.2,1.1);
\draw [shift={(0,0)},color=qqwuqq,fill=qqwuqq,fill opacity=0.1] (0,0) -- (0:0.07) arc (0:53.13:0.07) -- cycle;
\draw(0,0) circle (1cm);
\draw [color=qqqqff] (0.6,0.8)-- (0.6,-0.8);
\draw (0,0)-- (0.6,0.8);
\draw (0,0)-- (0.6,-0.8);
\draw [dash pattern=on 1pt off 1pt,color=qqzzqq,domain=-1.2:1.2] plot(\x,{(-0-0*\x)/1.6});
\draw [color=zzqqff] (0,0)-- (0.6,0);
\draw [color=zzqqff] (0,0) circle (0.6cm);
\draw [color=qqzzqq] (0,0)-- (-0.22,0.29);
\draw [color=qqzzqq] (-0.12,0.15) -- (-0.1,0.15);
\draw [color=qqzzqq] (-0.12,0.15) -- (-0.12,0.13);
\draw [color=qqzzqq] (0,0)-- (-0.22,-0.29);
\draw [color=qqzzqq] (-0.12,-0.15) -- (-0.12,-0.13);
\draw [color=qqzzqq] (-0.12,-0.15) -- (-0.1,-0.15);
\begin{scriptsize}
\fill [color=uququq] (0,0) circle (0.5pt);
\draw[color=uququq] (0.01,-0.09) node[font=\scriptsize] {$o$};
\fill [color=qqqqff] (0.6,0.8) circle (0.5pt);
\draw[color=qqqqff] (0.65,0.95) node[font=\scriptsize] {$v_i$};
\fill [color=qqqqff] (0.6,-0.8) circle (0.5pt);
\draw[color=qqqqff] (0.65,-0.95) node[font=\scriptsize] {$v_j$};
\draw[color=qqzzqq] (-1.04,0.15) node[left,font=\scriptsize] {$H_{v_i-v_j}$};
\fill [color=qqzzqq] (0.6,0) circle (0.5pt);
\draw[color=qqzzqq] (0.8,0.07) node[font=\scriptsize] {$m_{ij}$};
\draw[color=zzqqff] (0.3,-0.1) node[font=\scriptsize] {$\nu^{n}_1$};
\draw[color=qqwuqq] (0.25,0.1) node[font=\scriptsize] {$\theta^\prime/2$};
\fill [color=qqzzqq] (-0.22,-0.29) circle (0.5pt);
\draw[color=qqzzqq] (-0.11,-0.4) node[font=\scriptsize] {$v_i^\prime$};
\fill [color=qqzzqq] (-0.22,0.29) circle (0.5pt);
\draw[color=qqzzqq] (-0.15,0.4) node[font=\scriptsize] {$v_j^\prime$};
\draw[color=qqzzqq] (-0.25,0.12) node[font=\scriptsize] {$e_j^\prime$};
\draw[color=qqzzqq] (-0.25,-0.12) node[font=\scriptsize] {$e_i^\prime$};
\end{scriptsize}
\end{tikzpicture}

%% file: images/tikz/intersection.tex
\begin{tikzpicture}[scale=.7]

\begin{axis}[
    axis equal,
    axis lines=middle,
    xmin=-1.5, xmax=1.5,
    ymin=0, ymax=2,
    ticks=none,
]

\addplot[thick,domain=0:180, samples=200]
({cos(x)}, {sin(x)});

\addplot[thick] coordinates {(-1,0) (-1,2)};
\addplot[thick] coordinates {(1,0) (1,2)};

\addplot[
    domain=-1:1,
    samples=200,
    pattern=north east lines,
    draw=none
]
({x},{sqrt(1-x^2)}) -- (1,2) -- (-1,2) -- cycle;
\end{axis}
\end{tikzpicture}

%% file: biblio.bib
@article{Jacquemet_inradius-hyperbolic-simplex_2014,
    AUTHOR = {Jacquemet, Matthieu},
     TITLE = {The inradius of a hyperbolic truncated {$n$}-simplex},
   JOURNAL = {Discrete Comput. Geom.},
  FJOURNAL = {Discrete \& Computational Geometry. An International Journal
              of Mathematics and Computer Science},
    VOLUME = {51},
      YEAR = {2014},
    NUMBER = {4},
     PAGES = {997--1016},
      ISSN = {0179-5376,1432-0444},
   MRCLASS = {52B11 (51F15 51M25 52C17)},
  MRNUMBER = {3216674},
       DOI = {10.1007/s00454-014-9600-y},
       URL = {https://doi.org/10.1007/s00454-014-9600-y},
}

@article{Haagerup-Munkholm_hyperbolic-simplices-max-vol_1981,
    AUTHOR = {Haagerup, Uffe and Munkholm, Hans J.},
     TITLE = {Simplices of maximal volume in hyperbolic {$n$}-space},
   JOURNAL = {Acta Math.},
  FJOURNAL = {Acta Mathematica},
    VOLUME = {147},
      YEAR = {1981},
    NUMBER = {1-2},
     PAGES = {1--11},
      ISSN = {0001-5962,1871-2509},
   MRCLASS = {53C65 (53A35)},
  MRNUMBER = {631085},
MRREVIEWER = {Troels\ J\o rgensen},
       DOI = {10.1007/BF02392865},
       URL = {https://doi.org/10.1007/BF02392865},
}

@book{Ratcliffe_Foundations-Hyperbolic-Manifolds_2019,
    AUTHOR = {Ratcliffe, John G.},
     TITLE = {Foundations of hyperbolic manifolds},
    SERIES = {Graduate Texts in Mathematics},
    VOLUME = {149},
   EDITION = {Third},
 PUBLISHER = {Springer, Cham},
      YEAR = {[2019] \copyright 2019},
     PAGES = {xii+800},
      ISBN = {978-3-030-31597-9; 978-3-030-31596-2},
   MRCLASS = {57M50 (20H10 30F40 57K32)},
  MRNUMBER = {4221225},
       DOI = {10.1007/978-3-030-31597-9},
       URL = {https://doi.org/10.1007/978-3-030-31597-9},
}

@Article{Bestvina_degenerations-hyperbolic-space_1988,
 Author = {Bestvina, Mladen},
 Title = {Degenerations of the hyperbolic space},
 FJournal = {Duke Mathematical Journal},
 Journal = {Duke Math. J.},
 ISSN = {0012-7094},
 Volume = {56},
 Number = {1},
 Pages = {143--161},
 Year = {1988},
 Language = {English},
 DOI = {10.1215/S0012-7094-88-05607-4},
 zbMATH = {4064062},
 Zbl = {0652.57009}
}

@article {Bonahon-bouts-des-varietes-hyperboliques-de-dimension-3,
    AUTHOR = {Bonahon, Francis},
     TITLE = {Bouts des vari\'et\'es hyperboliques de dimension {$3$}},
   JOURNAL = {Ann. of Math. (2)},
  FJOURNAL = {Annals of Mathematics. Second Series},
    VOLUME = {124},
      YEAR = {1986},
    NUMBER = {1},
     PAGES = {71--158},
      ISSN = {0003-486X,1939-8980},
   MRCLASS = {57N10 (30F40 32G15)},
  MRNUMBER = {847953},
MRREVIEWER = {G.\ Peter\ Scott},
       DOI = {10.2307/1971388},
       URL = {https://doi.org/10.2307/1971388},
}

@article{Peyrimohoff_simplices-minimal-edge-length-hyperbolic_2002,
author = {Peyerimhoff, Norbert},
title = {Simplices of Maximal Volume or Minimal Total Edge Length in Hyperbolic Space},
journal = {Journal of the London Mathematical Society},
volume = {66},
number = {3},
pages = {753-768},
doi = {https://doi.org/10.1112/S0024610702003629},
url = {https://londmathsoc.onlinelibrary.wiley.com/doi/abs/10.1112/S0024610702003629},
year = {2002},
}

@article{Edmonds-Allan-Martini_Orthocentric-simplices-centers_2005,
    AUTHOR = {Edmonds, Allan L. and Hajja, Mowaffaq and Martini, Horst},
     TITLE = {Orthocentric simplices and their centers},
   JOURNAL = {Results Math.},
  FJOURNAL = {Results in Mathematics},
    VOLUME = {47},
      YEAR = {2005},
    NUMBER = {3-4},
     PAGES = {266--295},
      ISSN = {1422-6383,1420-9012},
   MRCLASS = {51M04 (51M20)},
  MRNUMBER = {2153497},
MRREVIEWER = {Hans\ Havlicek},
       DOI = {10.1007/BF03323029},
       URL = {https://doi.org/10.1007/BF03323029},
}

@article{Edmonds-Allan-Martini_Coincidence-simplices-centers_2005,
    AUTHOR = {Edmonds, Allan L. and Hajja, Mowaffaq and Martini, Horst},
     TITLE = {Coincidences of simplex centers and related facial structures},
   JOURNAL = {Beitr\"age Algebra Geom.},
  FJOURNAL = {Beitr\"age zur Algebra und Geometrie. Contributions to Algebra
              and Geometry},
    VOLUME = {46},
      YEAR = {2005},
    NUMBER = {2},
     PAGES = {491--512},
      ISSN = {0138-4821},
   MRCLASS = {51M04 (15A99 51M20 51N20 52A20)},
  MRNUMBER = {2196932},
MRREVIEWER = {Hans\ Havlicek},
}

@article{Lemoine_tetraedres-equifacetaux_1880,
     author = {Lemoine, Em.},
     title = {Quelques th\'eor\`emes sur les t\'etra\`edres dont les ar\^etes oppos\'ees sont \'egales deux \`a deux, et solution de la question 1272},
     journal = {Nouvelles annales de math\'ematiques : journal des candidats aux \'ecoles polytechnique et normale},
     pages = {133--138},
     year = {1880},
     publisher = {Carilian-Goeury et Vor Dalmont},
     volume = {2e s{\'e}rie, 19},
     language = {fr},
     url = {https://www.numdam.org/item/NAM_1880_2_19__133_1/}
}

@Misc{BD-CLS_actions-Hinfini_2026,
    author = {Duchesne, Bruno and Simon, Christopher-Lloyd},
    title = {The variety of actions on hyperbolic spaces of infinite dimensions},
    note = {In prepration},
    year = {2026},
}
